\documentclass[12pt]{article}
\usepackage[a4paper,left=30mm,right=30mm,top=20mm,bottom=20mm]{geometry}
\usepackage[T1]{fontenc}
\usepackage{lmodern,amsmath,amsfonts,amssymb,amsthm,mathtools}
\usepackage{enumitem,graphicx,xcolor,mathrsfs,setspace,thmtools,microtype}
\usepackage[hidelinks]{hyperref}
\newcommand{\TV}{\mathrm{TV}}
\newcommand{\Wb}{W_1^b}
\newcommand{\proj}{\mathrm{proj}}
\newcommand{\Lip}{\operatorname{Lip}}

\newcommand{\pru}{\operatorname{pr}_{\uphor}}

\numberwithin{equation}{section}

\newlist{legal}{enumerate}{10}
\setlist[legal]{label*=\arabic*.}

\newif\ifproofread

\DeclarePairedDelimiter\abs{\lvert}{\rvert}

\newcommand\norm[1]{\left\lVert#1\right\rVert}

\makeatletter
\newcommand*\bigcdot{\mathpalette\bigcdot@{.5}}
\newcommand*\bigcdot@[2]{\mathbin{\vcenter{\hbox{\scalebox{#2}{$\m@th#1\bullet$}}}}}
\makeatother

\DeclarePairedDelimiter\ceil{\lceil}{\rceil}
\DeclarePairedDelimiter\floor{\lfloor}{\rfloor}

\newtheorem*{Lemma*}{Lemma}
\newtheorem*{Corollary*}{corollary}
\newtheorem*{theorem*}{Theorem}
\newtheorem*{definition*}{Definition}

\newtheorem{theorem}{Theorem}[section]
\newtheorem{nota}[theorem]{Notation}

\newtheorem{definition}[theorem]{Definition}

\newtheorem{remark}[theorem]{Remark}
\newtheorem{lemma}[theorem]{Lemma}

\newtheorem{proposition}[theorem]{Proposition}
\newtheorem{corollary}[theorem]{Corollary}
\newtheorem{claim}[theorem]{Claim}

  \newtheorem{theoremalpha}{Theorem}

\newcommand{\thistheoremname}{}
\newtheorem*{genericthm}{\thistheoremname}

\newcommand{\NN}{\mathbb{N}}

\newcommand{\RR}{\mathbb{R}}
\newcommand{\ind}{\mathbf{1}}

\newcommand{\eps}{\epsilon}

\newcommand{\subd}{\on{Sub}_d(G)}

\newcommand{\Ff}{\mathcal{F}}

\newcommand{\Ii}{\mathcal{I}}

\newcommand{\Mm}{\mathcal{M}}

\newcommand{\Ad}{\operatorname{Ad}}

\newcommand{\on}[1]{\operatorname{#1}}

\newcommand{\uphor}{\mathfrak{u}}
\newcommand{\pminus}{\mathfrak{p}^{-}}
\newcommand{\Lie}{\operatorname{Lie}}
\newcommand{\supp}{\operatorname{supp}}
\newcommand{\E}{\mathbb{E}}

\newcommand{\dX}{d_X}
\newcommand{\Grass}{\operatorname{Gr}}
\newcommand{\End}{\operatorname{End}}

\newcommand{\g}{\mathfrak{g}}
\newcommand{\aCartan}{\mathfrak{a}}

\DeclareMathOperator{\cov}{cov}

\newcommand{\inv}{\mathrm{inv}}

\newcommand{\pid}{\mathrel{\ooalign{$\lneq$\cr\raise.22ex\hbox{$\lhd$}\cr}}}

\AtBeginDocument{}

\DeclareMathOperator{\Stab}{Stab}

\title{Applications of Almost Stationarity I: Quantitative Growth of Injectivity Radius and the Stuck--Zimmer Theorem}
\date{}
\author{
  Ilya Gekhtman \footnote{\texttt{ilyagekh@gmail.com}}, Simon Machado \footnote{\texttt{smachado@ethz.ch}}, Omri Solan \footnote{\texttt{os4293@princeton.edu}} and Yuval Yifrach \footnote{\texttt{yuval.yifrach@math.uzh.ch}}
}
\begin{document}
\maketitle

\begin{abstract}
Fraczyk and Gelander proved that for a higher-rank simple Lie group $G$
and a discrete subgroup $\Gamma\leq G$ of infinite covolume, the associated
locally symmetric space has points of arbitrarily large injectivity radius,
resolving a conjecture of Margulis. We obtain an explicit lower bound on
the growth of the maximal injectivity radius on growing balls: for all
sufficiently large $R$, a ball of radius $c\log^{(4)}R$ can be embedded
with its centre at distance at most $R$ from a fixed basepoint.
A definite proportion of the corresponding heat-kernel averages has
injectivity radius at least this large. Iterating the argument gives
proportions tending to one at further iterated-logarithmic scales, proving Benjamini--Schramm convergence of higher-rank lattices with a rate. 

As a step in the proof, we show that a discrete subgroup is a lattice if
$G/\Gamma$ carries a sufficiently almost-invariant probability measure
in bounded-Lipschitz distance, with a threshold depending only on $G$.
An observable formulation also gives a quantitative covolume bound.
The proofs use the structure of almost stationary measures developed in
the companion paper, together with geometric considerations.
Since our results do not rely on the Stuck--Zimmer theorem, our approach also yields a new proof of the Stuck--Zimmer theorem for ergodic
probability-preserving actions of higher-rank centre-free simple Lie groups.
\end{abstract}

\clearpage
{\small\singlespacing
\setcounter{tocdepth}{1}
\tableofcontents
}
\clearpage

\section{Introduction}

\subsection{Rigidity without limits}

A large part of the rigidity theory of higher-rank Lie groups is proved by
one and the same scheme. From a geometric or algebraic situation one
manufactures a family of probability measures on an auxiliary space (such as the Chabauty space $\on{Sub}(G)$ of closed subgroups of $G$), one passes
to a weak limit to obtain a measure which is invariant, or at least
stationary, and one then appeals to a structure theorem to identify that
limit. The structure theorems of Nevo--Zimmer \cite{NZ} and of
Stuck--Zimmer \cite{Stuck-Zimmer} are the two workhorses, and the scheme is
responsible for some of the most striking results in the field: the
convergence of normalised Betti numbers along sequences of lattices
\cite{7s17,ABBG}, the growth of torsion homology \cite{ABFG,Fra22}, and the
theorem of Fraczyk and Gelander \cite{FG} that a locally symmetric space of
higher rank either has finite volume or has points of arbitrarily large
injectivity radius.

The scheme is also lossy. The limit measure records only what
survives at infinity, and the compactness argument gives no control on
\emph{how fast} it was approached. Consequently these proofs give no
quantitative rate.

Making rigidity effective has been a central theme in homogeneous dynamics
over the last decade: effective density and equidistribution for unipotent
flows \cite{LM14,LMW,Yang,LMWY,LMWY-Oppenheim}, an effective closing lemma
\cite{LMMSW}, effective equidistribution of closed semisimple orbits
\cite{EMV}, and, for random walks, the effective results of
Bourgain--Furman--Lindenstrauss--Mozes \cite{BFLM} on the torus and of
B\'enard--He \cite{BH24,BH25} on homogeneous spaces. It is worth recording
that effective results are not merely strengthenings of their qualitative counterparts. For instance in \cite{BH25},
the argument is independent of Benoist--Quint and yields, as a by-product, a
new proof of their stationary measure rigidity theorem in the simple case.
The structure theory of stationary measures for actions of higher-rank
groups --- the Nevo--Zimmer and Stuck--Zimmer side of the subject --- has
so far stayed outside this development.

This is the first application paper in a series whose purpose is to bring
it inside, by removing the limit from the main rigidity argument. The companion paper \cite{QNZ} develops
a structure theory for measures which are only \emph{almost} stationary ---
measures $\nu$ with $\norm{\mu*\nu-\nu}_{\TV}\leq\eps$ --- and proves that
the Nevo--Zimmer dichotomy persists for them, with explicit rates in $\eps$.
Such measures are easy to come by: the
Ces\`aro averages
\begin{equation}\label{eq: cesaro intro}
    \nu_n=\frac{1}{n}\sum_{i=1}^{n}\mu^{*i}*\delta_x
\end{equation}
satisfy $\norm{\mu*\nu_n-\nu_n}_{\TV}\leq 2/n$ for \emph{every} $n$ and every
$x$, with no limit taken and no compactness invoked.

The present paper is the geometric half of the programme. Our aim is to show
that the effective structure theory of \cite{QNZ} is strong enough to carry
the geometric arguments that previously required the qualitative theorems. Going further, it does so with three kinds of gain. It gives quantitative
strengthenings of theorems previously accessible only in qualitative form
(Theorem~\ref{thm: qfg consequence}); it produces statements which have no qualitative
counterpart, because they are statements about approximate objects
(Theorem~\ref{thm: B}); and it gives shorter proofs of
the qualitative theorems themselves (Theorem~\ref{thm: C}).

\subsection{Covolume control and effective growth of the injectivity radius}

Let $G$ be a connected simple Lie group with finite centre and real rank at
least $2$, let $K<G$ be maximal compact and $X=K\backslash G$ the associated
symmetric space, and let $\Gamma\leq G$ be discrete. Recall that a closed
subgroup $\Lambda\leq G$ is \emph{confined} if the trivial subgroup is not a
Chabauty limit of conjugates of $\Lambda$. For a discrete subgroup of a
centre-free semisimple Lie group this happens precisely when the injectivity
radius of $X/\Lambda$ is uniformly bounded
\cite[\S7]{BGL}. In this language, the theorem of Fraczyk and Gelander
\cite{FG} reads: \emph{a confined discrete subgroup of a higher-rank simple
centre-free Lie group is a lattice}. For a finite-centre group,
we apply this formulation to the image in the adjoint quotient. It resolves a conjecture of Margulis and implies his normal subgroup theorem - the reason the
statement fails in rank one is that an infinite-index infinite normal
subgroup $N$ of a cocompact lattice $\Delta$ is a confined non-lattice.
Indeed, normality makes its conjugacy orbit a continuous image of the
compact space $G/\Delta$, and this image does not contain the trivial
subgroup. In higher rank the normal subgroup theorem forbids such an $N$. The confined-subgroup formulation has since been
pushed further by Bader--Gelander--Levit \cite{BGL}, who remove the property
(T) assumption and show that a confined subgroup of a higher-rank
irreducible lattice has finite index.

All of this is qualitative: confinement is an existence statement about
\emph{some} bound, with no control on the bound. In
\cite[Remark 1.4(ii)]{FG} the authors ask for a quantitative version. Our
first result answers this and in fact goes well beyond.

We show an alternative: either the subgroup at hand has small covolume, or almost all conjugates drawn at random with respect to the heat-kernel average have quantitatively large injectivity radius. Precisely:

\begin{theoremalpha}\label{thm A}
Let \(G\) be a connected centre-free simple Lie group of real rank at least two,
let \(K<G\) be maximal compact, and put \(X=K\backslash G\).
Let \(M=X/\Gamma\), where \(\Gamma\leq G\) is discrete, and fix
\(o\in M\) with \(\operatorname{inj}(o)>0\). Put
\(\iota:=\min\{1,\operatorname{inj}(o)\}\).
Write \(\nu_n^o\) for the Ces\`aro average of the heat distributions from
\(o\) at times \(N_0,2N_0,\ldots,nN_0\), where \(N_0=N_0(G)\) is fixed
in \eqref{def: the measure mu}.

There are constants \(a_G,c_G,C_G>0\) and \(0<p_G<1\), depending only on
\(G\), with the following property. For every \(0<\varepsilon<1\), there
exist an integer \(m=O_G(\varepsilon^{-8})\) and a constant
\(C_{G,\varepsilon}>0\) such that, for every integer
\[
 n\geq2\exp^{(m-1)}\!\left(C_{G,\varepsilon}\,\iota^{-a_G}\right),
\]
at least one of the following holds:
\begin{enumerate}[label=\arabic*.]
\item \(M\) has finite volume, with
\[
 \operatorname{vol}(M)\leq C_Ge^{C_Gn}.
\]
\item One has
\[
 \nu_n^o\!\left\{x\in M:
       \operatorname{inj}(x)\geq c_G\log^{(m)}n\right\}
 \geq1-\varepsilon.
\]
\end{enumerate}
Here \(\log^{(j)}\) and \(\exp^{(j)}\) denote \(j\)-fold iteration.
For \(1-p_G\leq\varepsilon<1\), one may take \(m=4\).
\end{theoremalpha}

In the proof, \(m=3k+1\), where \(k\) is the number of stages allowed in
the extraction of projective pieces. The case \(k=1\) gives a fixed
positive proportion, and letting \(k\) grow gives proportions tending to
one. At a singular basepoint the same conclusion holds with
\(\iota:=\min\{1,\Ii(g_0\Gamma g_0^{-1})\}\), where \(g_0\) is any lift
of \(o\); see the deduction at the end of
Section~\ref{sec: almost stationarity in g mod gamma}.

While Theorem \ref{thm A} might be difficult to grasp at first sight, specific choices of the parameters recover the well-known theorems mentioned in the introduction. 

\subsubsection{A growth gap for locally symmetric spaces and the Fraczyk--Gelander theorem}
Take for instance $k=1$. A locally symmetric space of finite volume has
bounded injectivity radius, since the volume of an embedded ball of
radius $R$ tends to infinity with $R$. Theorem~\ref{thm: qfg consequence}
therefore says that no intermediate behaviour is possible:

\begin{corollary}[Strengthened Fraczyk--Gelander theorem \cite{FG}]\label{cor: gap}
	With $G$, $X$, $\Gamma$, $o$ as above, exactly one of the following
	holds:
	\begin{enumerate}
		\item $\Gamma$ is a lattice, and $\sup_{X/\Gamma}\on{inj}<\infty$;
		\item $\displaystyle\liminf_{r\to\infty}\;
		      \frac{\sup_{p\in B^{X/\Gamma}_r(o)}\on{inj}(p)}{\log^{(4)}r}
		      \;>\;0 .$
	\end{enumerate}
	Equivalently: a discrete subgroup $\Gamma\leq G$ whose maximal
	injectivity radius on $B_r(o)$ is $o(\log^{(4)}r)$ along some sequence
	$r\to\infty$ is a lattice.

\end{corollary}
 Going further, Theorem \ref{thm A} asserts that in case $(b)$ the Cesaro average of heat-kernels spends most of its time in the thick part of the quotient $G/\Gamma$.

We do not expect $\log^{(4)}$ to be optimal. Two comparisons put this in perspective however. On the one hand, in the analogous situation for
graphs, Ab\'ert--Glasner--Vir\'ag \cite{AGV-Kesten,AGV-Kesten2} prove that,
for each fixed degree $d\geq3$, there is $\beta_d>0$ such that in a finite
$d$-regular Ramanujan graph $\mathcal G$, all but an
$O_d((\log|\mathcal G|)^{-\beta_d})$ proportion of vertices have a tree
neighbourhood of radius $\beta_d\log^{(2)}|\mathcal G|$. Thus two logarithms
arise.

On the other hand, the related question-how large
must the injectivity radius of a congruence cover be?- has a
well-developed quantitative theory going back to Buser--Sarnak \cite{BS94},
with logarithmic-in-volume systole bounds and, in several cases, sharp
constants \cite{KSV,Murillo,LLM}. These constants depend on the group and
the family of covers, and determining sharp values is part of that theory.
In the opposite direction, Anosov subgroups provide the natural
supply of infinite-covolume examples in higher rank, and their quotients
have unbounded injectivity radius \cite{KLP}.
\subsubsection{Benjamini--Schramm convergence}

On the other end of the spectrum, letting $k$ tend to infinity with $n$, we can show that most points in the quotient space of lattices of growing covolume must have large injectivity radius. In other words, this establishes a rate at which locally symmetric spaces converge in the Benjamini--Schramm sense to the symmetric space $X$.  

\begin{theorem}[Uniform quantitative Benjamini--Schramm convergence, informal]
\label{thm: quantitative bs informal}
Let $G$ be a connected centre-free simple Lie group with
$\mathrm{rank}_{\RR}G\geq2$, let $X=K\backslash G$, and let
$\Gamma\leq G$ be a lattice. Put $M=X/\Gamma$ and
$V=\operatorname{vol}(M)$, with the fixed metric normalisation.
For every fixed $R$,
\[
 \frac{\operatorname{vol}(M_{\leq R})}{V}
 =O_G\bigl((\log^*V)^{-1/8}\bigr)
 \qquad(V\to\infty),
\]
where $M_{\leq R}:=\{p\in M:\operatorname{inj}_M(p)\leq R\}$,
$\log^*t$ counts the logarithms needed to reach a value at most $e$,
and the threshold for $V$ may depend on $G$ and $R$.
\end{theorem}

A fully quantitative statement, including the dependence on $R$, is given in Theorem \ref{thm: quantitative bs} below. 

While the rate of convergence here is particularly slow (one is drowning in $\log$s), it is worthwhile comparing this result with the state of the art. Results concerning congruence covers show that one could expect, under number-theoretic assumptions, a single logarithm of the covolume. Our bound applies to \emph{all} lattices of a given simple Lie group, and that is precisely its strength. We furthermore expect to be able to derive quantitative estimates for the growth of many lattice invariants much inspired by the original applications of the Benjamini--Schramm convergence to the symmetric space from \cite{7s17,ABBG}.
\subsubsection{The Stuck--Zimmer theorem}\label{Subsec: A short proof}

As a third highlighted application we can record a new approach to the well-known Stuck--Zimmer theorem. 

\begin{theoremalpha}[Stuck--Zimmer \cite{Stuck-Zimmer}]\label{thm: C}
	Let $G$ be a connected centre-free simple Lie group with
	$\mathrm{rank}_\RR G\geq 2$ and let $G\curvearrowright(X,\nu)$ be an ergodic,
	measure-preserving action on a standard probability space. Then the
	action is either essentially free or essentially transitive.
\end{theoremalpha}

The classical argument uses the intermediate factor theorem
\cite{NZ-IFT} and an analysis of amenable actions.
In Section~\ref{sec: stuck zimmer} we give a short proof
from the quantitative growth theorem: it is a direct derivation from the effective
structure theorem of \cite{QNZ}. We emphasise that the proof of that theorem
in \cite{QNZ} is independent of \cite{Stuck-Zimmer}, so no circularity is involved.

The reason a quantitative tool gives a shorter qualitative proof is
explained in Section~\ref{sec: methods intro}, and is, we think, the most
interesting structural feature of the method.

Our argument, like the original, uses property~$(T)$, and it is worth
recalling how much is known without it. Hartman and Tamuz \cite{HT} showed
that property~$(T)$ for a single simple factor suffices; the Stuck--Zimmer
conjecture, that the theorem holds for every higher-rank semisimple group
regardless of property~$(T)$, remains open \cite[Question 2.1]{Gel-survey},
with substantial partial results obtained by Bader--Shalom \cite{BSh},
Creutz--Peterson \cite{CP}, Bader--Boutonnet--Houdayer--Peterson \cite{BBHP}
and, most recently, Bader--Gelander--Levit \cite{BGL}. We do not know
whether the method of this paper has anything to say in that direction.

\subsubsection{Sharpness of the bounds}
As already mentioned, we do not expect the bounds arising from Theorem \ref{thm A} to be sharp. Let us however locate where in the proof each logarithm is lost. 

Start with the case $k =1$. The four logarithms can be read from the scales in the proof. Put
$u:=\log|\log\eps|$. The choice of the projective-factor parameter
gives a norm radius $R$ which is a fixed power of $u$, while the product
of its error and radius tends to zero. The almost-factor lemma then gives
a norm discreteness radius which is a power of $\log R$, hence of
$\log^{(3)}(1/\eps)$. Passing to Riemannian distance takes one more
logarithm. Since $\eps=2/n$, the resulting scale is $\log^{(4)}n$.
For the general case $k > 1$, Theorem~\ref{thm:genesis} simply performs the same calculation after several
successive extractions.

\subsection{A criterion for being a lattice}
The proof of Theorem~\ref{thm A} passes through a statement which we believe is of
independent interest: it characterises lattices by the existence of an
\emph{approximately} invariant measure. Below, $G_1$ is the unit ball of $G$, and for
a measurable $f$ on $G/\Gamma$ we write $W_1^{f}$ for the distance between two measures as
seen through $f$, that is, the supremum of $|\nu_1(\phi\circ f)-\nu_2(\phi\circ f)|$
over $1$-Lipschitz $\phi$ bounded by $1$ (see \eqref{def: observable metric}).

\begin{theoremalpha}\label{thm: B}
Let $G$ have property~$(T)$. There is a threshold $\eps_{\mathrm T}=\eps_{\mathrm T}(G)>0$
with the following property. Let $\Gamma\leq G$ be discrete and let
$\nu$ be a probability measure on $G/\Gamma$ with density $\rho>0$
almost everywhere relative to $m_{G/\Gamma}$. For $g\in G_1$, put
$f_{\nu,g}:=d\nu/d(g\nu)$. If $0\leq\eps<\eps_{\mathrm T}$ and
\begin{equation}\label{eq: thmB hypothesis}
 W_1^{f_{\nu,g}}(g\nu,\nu)\leq\eps\qquad(g\in G_1),
\end{equation}
then $\Gamma$ is a lattice. If, moreover, a compact set
$C\subseteq G/\Gamma$ satisfies $\nu(C)\geq1/2$, then
\[
 m_{G/\Gamma}(G/\Gamma)\leq4m_{G/\Gamma}(C).
\]
Conversely, a lattice carries an invariant probability measure, for which
\eqref{eq: thmB hypothesis} holds with $\eps=0$.
\end{theoremalpha}

One direction is trivial: a lattice carries an invariant probability.
The content is the converse, which says that below a threshold depending
only on $G$, an approximately invariant measure of this kind can exist
only if an invariant one does. The choice of the observable
is the point: $f_{\nu,g}$ records the failure of $\nu$ to be invariant under $g$.

To explain the proof, write $q(x):=\rho(g^{-1}x)$ and take
\[
 \phi(t):=\frac{t-1}{t+1}\qquad(t>0).
\]
This function is bounded by $1$ and is $2$-Lipschitz. Testing
\eqref{eq: thmB hypothesis} against $\phi/2$ gives
\[
 \int\frac{(\rho-q)^2}{\rho+q}\,dm_{G/\Gamma}\leq2\eps.
\]
The integrand bounds $(\sqrt\rho-\sqrt q)^2$ from above. Thus the
unit vector $\sqrt\rho\in L^2(m_{G/\Gamma})$ is almost invariant,
and property~$(T)$ supplies a nonzero invariant vector. This forces
$m_{G/\Gamma}$ to be finite. Keeping track of the orthogonal projection
onto invariant vectors gives the covolume assertion.

There is also a formulation in the bounded variant of the $1$-Wasserstein distance. We record it in the introduction as it has independent interest as well as illustrative value:

 \begin{corollary}\label{cor: weak lattice criterion, intro}
Assume that $G$ has property~$(T)$. There is $\eps_{\mathrm w}=\eps_{\mathrm w}(G)>0$ such that, for every
discrete $\Gamma\leq G$, the existence of a probability measure $\tau$ on
$G/\Gamma$ satisfying
\[
 W_1^b(\tau,g\tau)\leq\eps_{\mathrm w}\qquad(g\in G_1)
\]
implies that $\Gamma$ is a lattice.
\end{corollary}
 
The passage to Theorem~\ref{thm: B} uses convolution with a fixed
positive kernel whose logarithmic density is Lipschitz. This
regularization applies to arbitrary almost-invariant measures, including
the residual measures produced by the structure theorem.

\subsection{Method}\label{sec: methods intro}

We sketch the proof of Theorem~\ref{thm:genesis}; Theorems~\ref{thm: B}
and~\ref{thm: C} come out of the same two lemmas.

Let $Y:=G/\Gamma$ and take the Ces\`aro averages
\[
 \nu_n:=\frac1n\sum_{i=1}^n\mu^{*i}*\delta_{[e]}.
\]
These measures are $2/n$-almost stationary. Working with $\nu_n$ itself
retains a rate throughout the argument. It also gives us direct analytic
control on the homogeneous space $G/\Gamma$: the almost-invariant
alternative can be analysed in $L^2(m_{G/\Gamma})$ using property~$(T)$.
There is no need to identify an abstract invariant measure obtained as a
weak limit.

The main theorem of \cite{QNZ}, in the form we use, says that for an
$\eps$-almost stationary measure $\nu$ and a prescribed observable $f$,
either
\begin{itemize}
 \item[(a)] $\nu$ is almost invariant under $G_1$ as seen through $f$; or
 \item[(b)] a definite proportion of its mass sits on points whose
 stabilisers, seen at a large scale, are approximately contained in
 proper parabolic subgroups.
\end{itemize}
The precise statement is Theorem~\ref{thm: qnz intro}. Its ergodic
form, Theorem~\ref{thm:ergodic-qnz}, extracts projective submeasures
successively while controlling the stationarity of the residual measure.
If the construction stops in the almost-invariant alternative, we
regularize that residual measure and apply Theorem~\ref{thm: B}.
This proves that $\Gamma$ is a lattice, and also bounds its covolume.

In the other alternative, Lemma~\ref{lem: almost factor} turns a
projective factor into a large discreteness radius after a further short
average. Its proof has two geometric steps. Contraction of the unipotent
radical, together with discreteness, first confines the subgroup near a
Levi factor. Repeating this for finitely many Levi factors with trivial
intersection then confines it near the identity. A common lower bound on
the discreteness radius makes this intersection trivial on a large ball.
The polynomial-rate Furstenberg decomposition supplies the almost
invariance needed for these steps.

We integrate this conclusion over the projective part and use the almost
stationarity of the original $\nu_n$ to return to it. Theorem~\ref{thm:genesis}
records the resulting estimate. One stage gives the fourfold logarithm
and a fixed positive proportion. More stages give a proportion tending
to one, at the cost of further iterated logarithms. Both conclusions
come from the same probability estimate.

For Stuck--Zimmer, we apply the case of one stage to a generic discrete
subgroup of an ergodic invariant random subgroup. If that subgroup is a
lattice, averaging on its finite-volume quotient gives the corresponding
lattice invariant random subgroup. Otherwise the quantitative estimate
puts a fixed positive mass on subgroups with discreteness radius tending
to infinity. Weak convergence then gives positive mass to the trivial
subgroup, and ergodicity forces the invariant random subgroup to be trivial.

\subsection{Outlook}

The structure theory of \cite{QNZ} is not specific to the injectivity
radius, and we regard Theorems~\ref{thm A}--\ref{thm: C} as a first
instalment rather than the intended endpoint. Three directions seem to us
the most immediate.

\emph{Effective convergence of lattice invariants.}
Theorem~\ref{thm: quantitative bs} gives a rate of Benjamini--Schramm
convergence for arbitrary lattices of growing covolume, and
Corollary~\ref{cor: bs tower threshold} makes the volume threshold explicit
for a prescribed radius and error. The passage from heat-kernel averages
to normalised volume uses invariance and the Margulis inequality.
We expect to derive corresponding rates for normalised Betti numbers
and other lattice invariants; these would require the additional
arguments in \cite{7s17,ABBG}.

For \emph{congruence} covers of a fixed arithmetic manifold, a strong
quantitative form of Benjamini--Schramm convergence, with Sarnak--Xue-type
consequences \cite{SX} for normalised Betti numbers, is already available
\cite{7s17}, and analogous statements are known through property~$(\tau)$
and the trace formula \cite{Levit-BS,Raimbault,Fraczyk-limit}.
These results use additional uniformity from congruence structure,
spectral estimates or the trace formula. Here the geometric threshold
applies to arbitrary lattices in the simple group considered, and the
next question is what rates it gives for their invariants.

\emph{Quantitative Stuck--Zimmer.} A further goal is to quantify
Theorem~\ref{thm: C}: under suitable control of small discreteness radii,
one expects an almost-invariant measure on $\subd$ to be close to a
convex combination of a lattice-supported invariant measure and
$\delta_{\{e\}}$. A precise statement and proof are left to future work.

\emph{Extending the structure theory.} Theorem~\ref{thm: qnz intro} is
stated for the heat kernel as step measure. Its stabiliser conclusion
is vacuous for free actions. Extending the choice of step measure and
finding a formulation that gives information for free actions are further
goals of the programme. The ergodic form is already used here, through successive extraction
of projective submeasures.

Finally, we note that the results of this paper belong to a broader effort
to replace qualitative invariants of lattices by quantitative ones ---
volume versus homotopy type \cite{Gel04}, volume versus rank
\cite{AGN,LS,FMW}, volume versus torsion \cite{ABFG,FHR}, and the
quantitative form of Selberg's lemma of Gelander--Slutsky \cite{GS-Selberg},
where a qualitative finiteness statement is likewise replaced by an explicit
bound in the covolume together with examples showing it to be nearly sharp.

\subsection{Structure of the paper}
Section~\ref{sec: def} fixes notation and states the main quantitative
alternative. Section~\ref{sec: almost stationarity in g mod gamma}
states the effective structure theorem and the two supporting lemmas,
proves Theorem~\ref{thm:genesis}, and deduces Theorem~\ref{thm A}.
Sections~\ref{sec: almost invariance} and~\ref{sec: almost factor} prove
the two lemmas. Section~\ref{sec: consequences} derives the
Fraczyk--Gelander, Benjamini--Schramm and Stuck--Zimmer consequences.
Appendix~\ref{app: geometric comparisons} records the geometric
comparisons. Appendix~\ref{app: heat-kernel computations} gives the
heat-kernel estimates and the regularization argument, and
Appendix~\ref{app: margulis} proves the convolution-power Margulis inequality.

\subsection{Acknowledgments}
We thank Emmanuel Breuillard, Manfred Einsiedler, Mikolaj Fraczyk, Alex Furman,
Tsachik Gelander, Segev Gonen Cohen, Alex Gorodnik, Arie Levit, Elon Lindenstrauss,
Amos Nevo, and Barak Weiss for many discussions during the preparation of this paper.
We also thank Yuval Gorfine and Michael Glasner for many long
discussions and for listening to many failed strategies. The last named
author would like to acknowledge the positive influence of L.R.\ on the
making of this project.

The first named author was supported by ISF grant 3423/24. The last named
author was supported by SNF grant 200020--212617.

Last but not least, the last named author would like to thank Shoshana Chaya
for her unconditional support and infinite patience. This paper is dedicated
to her with a lot of love.

\subsection{Statement on the use of AI assistants}

The authors used AI assistants based on large language models (LLMs) in the preparation of
this paper, as follows.

\emph{Writing and presentation.} AI assistants were used to help polish exposition, to suggest reorganizations of the material, and to
proofread. Every sentence was written and reviewed by the
authors.

\emph{Discussion.} The authors used AI assistants as a sounding board: to test
formulations, to look for alternative routes through arguments, and to
test claims. 

\emph{Mathematical content.} All mathematical ideas wiht one exception, statements and proofs in
this paper are the authors' own. The exception is Appendix \ref{app: margulis}, where we extend the inequality of Gelander-Margulis-Levit to measures of exponential moment. In this case, the authors came up with the idea to do this, and asked Chatgpt to write the text, which was then read line by line and corrected by the authors.
No other statement or proof was generated by artificial intelligence.

The authors have verified all results and take full responsibility for the
correctness and the content of this paper.

\section{Definitions and notation}\label{sec: def}

\subsection{Geometric notation}

Unless explicitly stated otherwise, $G$ is a connected noncompact simple
Lie group with finite centre. The rank assumption is stated whenever
the almost Nevo--Zimmer theorem is used. Haar measure on $G$ is
bi-invariant, since $G$ is unimodular.

We recall the following definition:
\begin{definition}[Locally Lipschitz orbits]\label{def: locally lischitz}
    Given an action of a Lie group $G$ with metric $d_G$ on a metric space $(X,d_X)$, we say that the action is by locally Lipschitz orbits if for every $x\in X$, the orbit map $o_x:G_1\rightarrow X$ defined on the $d_G$ unit ball in $G$, is $1$-Lipschitz.
\end{definition}

Fix a maximal compact subgroup $K<G$, a Haar measure $m_G$, and
normalised Haar measure $m_K$ on $K$. For each discrete $\Gamma\leq G$,
let $m_{G/\Gamma}$ be the quotient measure obtained using counting measure
on $\Gamma$, and write $\cov(G/\Gamma):=m_{G/\Gamma}(G/\Gamma)$.

Fix a Cartan inner product on
$\mathfrak g:=\Lie(G)$, invariant under $\Ad(K)$. Write
$\mathfrak g=\mathfrak k\oplus\mathfrak p$ for the corresponding Cartan
decomposition. We use this inner product and its norm on
$\mathfrak g$ and the associated operator norm on $\End(\mathfrak g)$.
Let $d_G$ be the right-invariant Riemannian metric defined by this
inner product. Left translation by $K$ and conjugation by $K$ are
isometries. The quotient metric on $X:=K\backslash G$ is the
associated symmetric-space metric; its basepoint is $o:=K$, and we
write $Kg$ for the image of $g$ in $X$.

With these choices fixed, ``depending only on $G$'' includes dependence
on the metric and Haar normalisations. When $G$ is centre-free, there
is a constant $b_G>0$, independent of the discrete subgroup $\Gamma$,
such that
\begin{equation}\label{eq: volume normalization}
 \operatorname{vol}(X/\Gamma)=b_G\cov(G/\Gamma).
\end{equation}
If $\Gamma$ is a lattice, the quotient of normalised Haar measure on
$G/\Gamma$ is therefore normalised Riemannian volume on $X/\Gamma$.

For $r>0$ and $R>1$, respectively, put
\[
 G_r:=\{g:d_G(e,g)<r\},\qquad
 G_R^{\|}:=\{g:|g|_{\Ad}<R\},
\]
where
\[
 |g|_{\Ad}:=\max\{\|\Ad(g)\|_{op},\|\Ad(g^{-1})\|_{op}\}.
\]
Thus $G_R^{\|}$ is a norm ball in the sense of this matrix size.
It is used only for $R>1$; small identity neighbourhoods are measured
with $d_G$. The balls $G_r$ and $G_R^{\|}$ are open; a bar denotes
their closure.
For a subgroup $H\leq G$, write $H_r:=H\cap G_r$ and
$H_R^{\|}:=H\cap G_R^{\|}$.

The matrix size is submultiplicative and $K$-bi-invariant. In
particular, both kinds of balls are invariant under conjugation by
$K$. All Lipschitz constants and Hausdorff distances on $G$ use
$d_G$, unless specified otherwise. On $G/\Gamma$ we use the quotient
metric
\[
 d_{G/\Gamma}(u\Gamma,v\Gamma)
 :=\inf_{\gamma\in\Gamma}d_G(u\gamma,v).
\]
Right invariance makes every orbit map $g\mapsto gx$ $1$-Lipschitz.
The comparison with the Cartan distance and with norm balls is
recorded in Lemma~\ref{lem: geometric comparisons}.

For a discrete $\Gamma<G$, define
\[
 \on{inj}_{X/\Gamma}(p)
 :=\sup\{r>0:\pi_\Gamma|_{B_r^X(\widetilde p)}
                    \text{ is injective}\},
\]
where $\pi_\Gamma:X\to X/\Gamma$ is the quotient map and
$\widetilde p$ is any lift of $p$. The value is $0$ if there is no
such positive $r$. This definition also applies when $\Gamma$ has
torsion.

When $G$ has finite centre, put $G_{\mathrm{ad}}:=G/Z(G)$,
$K_{\mathrm{ad}}:=K/Z(G)$ and let $\Gamma_{\mathrm{ad}}$ be the image of
$\Gamma$. The identification
\[
 K\backslash G/\Gamma
 \cong K_{\mathrm{ad}}\backslash G_{\mathrm{ad}}/\Gamma_{\mathrm{ad}}
\]
is an isometry. We use this effective action in all statements about
injectivity radius. The subgroup $\Gamma_{\mathrm{ad}}$ is discrete, and
is a lattice exactly when $\Gamma$ is. The geometric arguments which
assert triviality of a subgroup are carried out in the centre-free
group $G_{\mathrm{ad}}$, as indicated in their statements.

For a discrete subgroup $\Gamma<G$, define its discreteness radius by
\begin{equation}\label{def: disc rad}
 \Ii(\Gamma):=\sup\{r>0:\Gamma\cap G_r=\{e\}\}.
\end{equation}
We say that $\Gamma$ is $r$-discrete if
$\Gamma\cap G_r=\{e\}$, equivalently if $\Ii(\Gamma)\geq r$.
We write $\on{Sub}(G)$ for the compact Chabauty space of closed
subgroups, and $\subd=\on{Sub}_d(G)$ for its subspace of discrete
subgroups. We also write $\Gamma^g:=g\Gamma g^{-1}$.
For the trivial subgroup, $\Ii(\{e\})=+\infty$, and negative powers
of this value are interpreted as zero.

Given a metric space $(Y,d)$ and nonempty subsets $A,B\subseteq Y$,
the one-sided Hausdorff distance from $A$ to $B$ is
\begin{equation}\label{def: dh+}
 d_H^+(A,B):=\sup_{a\in A}d(a,B),\qquad
 d(a,B):=\inf_{b\in B}d(a,b).
\end{equation}

\subsection{Parabolic subgroups}

Our conventions follow those of~\cite[Section~2]{NZ}; this notation also
applies to connected semisimple Lie groups with finite centre. Fix a maximal
$\RR$-split torus
$S<G$ whose Lie algebra $\mathfrak a$ lies in $\mathfrak p$,
and choose a system $\Sigma^+$ of positive roots with associated
set $\Delta\subseteq\Sigma^+$ of simple roots. For
$\theta\subseteq\Delta$, set
\begin{equation}\label{def: nz setup}
    \mathfrak s_\theta
    :=
    \bigcap_{\alpha\in\theta}\ker\alpha,
\end{equation}
and let $S_\theta<S$ be the connected subgroup with Lie algebra
$\mathfrak s_\theta$. Define
\[
    L_\theta:=Z_G(S_\theta).
\]
Writing $V_\theta$ and $\overline V_\theta$ for the unipotent radicals
of the standard parabolic $P_\theta$ and its opposite
$\overline P_\theta$, respectively, we have
\[
    P_\theta=L_\theta\ltimes V_\theta,
    \qquad
    \overline P_\theta=L_\theta\ltimes\overline V_\theta.
\]
We also write
\[
    L_\theta=M_\theta S_\theta,
\]
where $M_\theta$ is the reductive Langlands factor with compact centre. For
$\theta=\varnothing$, we abbreviate
\[
    P:=P_\varnothing=MS\ltimes V,
    \qquad
    \overline P:=\overline P_\varnothing=MS\ltimes\overline V.
\]
Write $\mathfrak a^+:=\{H\in\mathfrak a:\alpha(H)>0\ (\alpha\in\Delta)\}$
for the open positive Weyl chamber. The Cartan projection
$\kappa:G\to\overline{\mathfrak a^+}$ is defined by
$g\in K\exp(\kappa(g))K$. Thus $d_X(K,Kg)=\|\kappa(g)\|$.
We reserve $\varkappa>0$ for the Margulis exponent and
$\kappa_G>0$ for a Kazhdan constant.

\subsection{Measures and quantitative invariance}

For a metric space $Y$, let $\Mm_1(Y)$ denote the space of Borel
probability measures on $Y$. If $G$ acts measurably on $Y$, we write
$g\nu$ for the pushforward of $\nu$ by $y\mapsto gy$. If $\mu$ is a
Borel probability measure on $G$, then
\[
    \mu*\nu:=\int_G g\nu\,d\mu(g).
\]

We use the following distances on probability measures
$\lambda_1,\lambda_2\in\Mm_1(Y)$:
\begin{enumerate}
    \item Their total variation distance is
    \begin{equation}
        d_{\mathrm{TV}}(\lambda_1,\lambda_2)
        :=
        \sup_{A\subseteq Y\text{ Borel}}
        \abs{\lambda_1(A)-\lambda_2(A)}.
    \end{equation}
    For a finite signed measure $\sigma$, its full variation norm is
    $\|\sigma\|_{\mathrm{TV}}:=\sup_{|f|\leq1}|\sigma(f)|$.
    Thus $\|\lambda_1-\lambda_2\|_{\mathrm{TV}}=2d_{\mathrm{TV}}(\lambda_1,\lambda_2)$.
    We use the full norm for stationarity estimates.

    \item Their bounded-Lipschitz distance is
    \begin{equation}\label{def: metric on measures}
        W_1^{b}(\lambda_1,\lambda_2)
        :=
        \sup\left\{
            \abs{\lambda_1(\varphi)-\lambda_2(\varphi)}:
            \on{Lip}(\varphi)\leq1,
            \ \norm{\varphi}_\infty\leq1
        \right\}.
    \end{equation}

    \item Given a measurable function $f:Y\to\RR$, define
    \begin{equation}\label{def: observable metric}
        W_1^{f}(\lambda_1,\lambda_2)
        :=
        W_1^{b}(f_*\lambda_1,f_*\lambda_2)
        =
        \sup_{\substack{\on{Lip}(\varphi)\leq1\\
                            \norm{\varphi}_\infty\leq1}}
        \abs{\lambda_1(\varphi\circ f)
             -\lambda_2(\varphi\circ f)}.
    \end{equation}

    \item If $\tau$ is a finite measurable partition of $Y$, define
    \begin{equation}
        d_\tau(\lambda_1,\lambda_2)
        :=
        \max_{A\in\sigma(\tau)}
        \abs{\lambda_1(A)-\lambda_2(A)}
        =
        \frac12\sum_{A\in\tau}
        \abs{\lambda_1(A)-\lambda_2(A)}.
    \end{equation}
\end{enumerate}
The metric $W_1^{b}$ metrizes weak convergence on $\Mm_1(Y)$ whenever $Y$
is Polish. The same suprema define $W_1^b$ and $W_1^f$ for finite
measures of equal mass, without normalisation; we use this convention for
the weighted residual measure in Theorem~\ref{thm:ergodic-qnz}.

Let $p_t$ be the $K$-bi-invariant heat-kernel probability density on
$G$ with respect to $m_G$, for the semigroup $e^{t\Delta_X}$.
We also use $p_t$ for the probability measure $p_t\,dm_G$ when writing
convolutions. Fix a sufficiently large integer
$N_0$, chosen in Lemma~\ref{lem: radius transfer}. When passing to
$G_{\mathrm{ad}}$, choose a common multiple of the times supplied for
$G$ and $G_{\mathrm{ad}}$, and use this same $N_0$ in both groups.
The drift inequality remains valid under further convolution powers.
Set
\begin{equation}\label{def: the measure mu}
    \mu:=p_1^{*N_0}=p_{N_0}.
\end{equation}

\begin{definition}[Almost stationarity]
Let $Y$ be a $G$-space. A probability measure $\nu\in\Mm_1(Y)$ is
\emph{$\eps$-almost stationary} with respect to $\mu$ if
\[
    \|\mu*\nu-\nu\|_{\mathrm{TV}}\leq\eps.
\]
\end{definition}

\begin{definition}[Almost invariance]\label{def: almost invariances}
Let $\nu\in\Mm_1(Y)$. We say that $\nu$ is
\begin{enumerate}
    \item \emph{norm-$\eps$-invariant} if
    \[
        d_{\mathrm{TV}}(g\nu,\nu)\leq\eps
        \qquad\text{for every }g\in G_1;
    \]

    \item \emph{weakly $\eps$-invariant} if
    \[
        W_1^{b}(g\nu,\nu)\leq\eps
        \qquad\text{for every }g\in G_1;
    \]

    \item \emph{$(f,\eps)$-almost invariant}, for a measurable
    $f:Y\to\RR$, if
    \[
        W_1^{f}(g\nu,\nu)\leq\eps
        \qquad\text{for every }g\in G_1;
    \]

    \item \emph{$(\tau,\eps)$-almost invariant}, for a finite
    measurable partition $\tau$ of $Y$, if
    \[
        d_\tau(g\nu,\nu)\leq\eps
        \qquad\text{for every }g\in G_1.
    \]
\end{enumerate}
For a subgroup $H\leq G$, almost invariance under $H$ is tested
on $H\cap G_1$. When the acting group is clear, we freely use equivalent phrases such
as ``$G$-weakly $\eps$-invariant'' and ``$f$-$\eps$-almost invariant.''
\end{definition}

\subsection{Lipschitz functions on orbits}

For $Y:=G/\Gamma$, with the quotient distance defined above, recall that its orbit maps are $1$-Lipschitz. Moreover, left translation by $g$ is Lipschitz,
with a constant bounded by $\|\Ad(g)\|$. Write
\[
 \Ff_1(Y):=\{f:Y\to\RR:\Lip(f)\leq1,\ \|f\|_\infty\leq1\}.
\]

We will also require Lipschitzity of a specific observable, which gives rise to the following definition: 

\begin{definition}[Log-Lipschitz functions]\label{def: loglip}
Let $Y$ be a $G$-space, let $f:Y\to(0,\infty)$, and let
$0<c_1\leq c_2<\infty$. We say that $f$ is
\emph{$(c_1,c_2)$-log-Lipschitz} if
\[
    c_1
    \leq
    \frac{f(x)}{f(g^{-1}x)}
    \leq
    c_2
    \qquad
    \text{for every }g\in G_1\text{ and }x\in Y.
\]
\end{definition}

\subsection{Statement of the main theorem}
We state the quantitative alternative underlying Theorem~\ref{thm A}. The rates are given by iterated logarithms and exponentials: we write $\log^{(j)}$ and $\exp^{(j)}$ for iterates, and set
$\log^*t:=\min\{j\geq0:\log^{(j)}t\leq e\}$.

\begin{theorem}[Growth of injectivity radius vs controlled covolume]\label{thm:genesis}
Let $G$ be a connected noncompact simple Lie group with trivial centre
and real rank at least two, and let $\Lambda\leq G$ be discrete.
There are $p_G,c_G,C_G,\beta>0$ and $P>0$, depending only on $G$.
Here, $\varkappa$ is as in Lemma~\ref{lem: margulis functions intro}, $\mu$ is fixed in \eqref{def: the measure mu}, $\beta$ is as in
Lemma~\ref{lem: almost factor}, and $v_G$ below is as in
Lemma~\ref{lem:concentration}.

Let $n,k\geq1$ be integers and suppose
the iterated logarithms are defined and satisfy
\[
 \log^{(3k)}(n/2)\geq C_G
       \bigl(1+\Ii(\Lambda)^{-\varkappa}\bigr)^{9(\dim G+1)/\varkappa} k^P.
\]
Put
\[
 o:=[e]_{G/\Lambda},\qquad
 \nu_n:=\frac1n\sum_{i=1}^n\mu^{*i}*\delta_o.
\]
Then at least one of the following holds:
\begin{enumerate}[label=\arabic*.]
\item $\Lambda$ is a lattice, and
\[
 \cov(G/\Lambda)
 \leq4m_G\!\left(\overline G_{v_G(n+C_G)}\right)
 \leq C_Ge^{C_Gn}.
\]
\item One has
\begin{equation}\label{eq:genesis-conclusion}
 \begin{aligned}
 &\nu_n\!\left(\left\{x\in G/\Lambda:
 \Stab_G(x)\cap
 G_{c_G(\log^{(3k)}(n/2))^\beta}^{\|}=\{e\}\right\}\right)\\
 &\hspace{28mm}\geq\max\{p_G,1-C_Gk^{-1/8}\}.
 \end{aligned}
\end{equation}
\end{enumerate}

\end{theorem}

The proof appears in Section~\ref{sec: almost stationarity in g mod gamma},
after the statements of the structural and geometric inputs.

\section{The structure of almost stationary measures on \texorpdfstring{$G/\Gamma$}{G/Gamma}}
\label{sec: almost stationarity in g mod gamma}
In this section we analyse almost stationary measures on $G/\Gamma$. Conditional on two lemmas (Lemmas \ref{lem: almost invariance} and \ref{lem: almost factor}) that we establish later, we establish our main theorem (Theorem \ref{thm:genesis}).

\subsection{Preliminaries on almost stationary measures}

The main object of study are the following averages:

\begin{definition}\label{def: nu n}
For $n\geq1$, put
\[
 \nu_n:=\frac1n\sum_{i=1}^n\mu^{*i}*\delta_{[e]_{G/\Gamma}},
 \qquad A_\mu^m\tau:=\frac1m\sum_{i=1}^m\mu^{*i}*\tau.
\]
The measure $\nu_n$ has a positive density relative to $m_{G/\Gamma}$.
\end{definition}

They are easily seen to be almost stationary.

\begin{claim}\label{cl: almost stationary}
The measure $\nu_n$ is $2/n$-almost stationary. More generally, if
$\|\mu*\tau-\tau\|_{\mathrm{TV}}\leq\eps$, then
\begin{equation}\label{eq:averaging}
 \|\mu*A_\mu^m\tau-A_\mu^m\tau\|_{\mathrm{TV}}\leq2/m,
 \qquad\|A_\mu^m\tau-\tau\|_{\mathrm{TV}}\leq m\eps.
\end{equation}
\end{claim}
\begin{proof}
Telescoping gives
\[
 \mu*\nu_n-\nu_n
 =\frac1n(\mu^{*(n+1)}-\mu)*\delta_{[e]}.
\]
Its full variation norm is at most $2/n$. The same computation gives
the first inequality in \eqref{eq:averaging}; the second follows from
$\|\mu^{*i}*\tau-\tau\|_{\mathrm{TV}}\leq i\eps$ and averaging.
\end{proof}

\begin{definition}[Almost factor]\label{def: almost factor}
Let $Y$ be a $G$-space, let $W\subseteq Y$ be measurable, and let
$R>1$ and $\delta>0$. We say that
$W$ has a projective $(R,\delta)$-factor if there is a measurable choice
of proper parabolic subgroups $x\mapsto Q_x$ such that
\[
 d_H^+(\Stab_G(x)\cap G_R^{\|},Q_x)\leq\delta\qquad(x\in W).
\]
A probability is $(R,\delta)$-projective if it gives full mass to such a set.
\end{definition}

We use the following theorem of \cite[Theorem~1.10]{QNZ}, with its
constants taken for the fixed heat-kernel time $N_0$.
\begin{theorem}[Almost Nevo--Zimmer]\label{thm: qnz intro}
There are $C,c,c'>0$, depending only on $G$, with the following property.
Let $G$ be connected, simple, noncompact, of finite centre and real rank
at least two. Let $Y$ be a $G$-space where $G$ acts by locally Lipschitz orbits (see Definition \ref{def: locally lischitz}), let $\nu$ be an
$\eps$-almost stationary probability measure for the fixed heat-kernel measure $\mu$,
where $0<\eps<e^{-1}$, and let $f\in\Ff_1(Y)$.
For every $\eta\in[0,1]$, at least one of the following holds:
\begin{enumerate}[label=\arabic*.]
\item $\nu$ is
$(f,C(|\log\eps|^{-c}+\eta^{1/4}))$-almost invariant;
\item for every $\zeta>|\log\eps|^{-1/40}$ there is a measurable
$W\subseteq Y$ with $\nu(W)\geq\eta$ having an
$(R_\zeta,\delta_\zeta)$-projective factor, where
\[
 R_\zeta:=|\log\eps|^{\zeta/(40\dim G)},\qquad
 \delta_\zeta:=C|\log\eps|^{c'\zeta/(40\dim G)}\zeta^c.
\]
\end{enumerate}
\end{theorem}

We also use the quantitatively weaker but qualitatively more precise \cite[Thm]{QNZ}.

The following decomposition theorem can be summarised along the following lines: any almost stationary measure is the sum of an almost invariant measure and an almost projective measure. However, a too simplistic statement would not specialise  to Theorem \ref{thm: qnz intro} when taking $k = 1$. 

An unsatisfactory consequence of that would be to get a  bound worse than $\log^{(4)}$ in Corollary \ref{cor: gap}.

\begin{theorem}[almost invariant vs projective decomposition]\label{thm:ergodic-qnz}
Let $\nu,Y,\eps$ satisfy the assumptions of Theorem~\ref{thm: qnz intro},
let $\Ff\subseteq\Ff_1(Y)$, let $k\geq1$ be an integer, and fix
$\vartheta\in(0,1]$. Set $\eps_0:=\eps$ and put
\[
 L_j:=\log^{(3j)}(1/\eps)\quad(0\leq j\leq k),\qquad
 U_k:=\log^{(3k-1)}(1/\eps),
\]
and $\eps_j:=L_j^{-1/2}$ for $1\leq j\leq k-1$.
Assume that the displayed iterated logarithms are defined and positive.
There are $a,b,C_1,T>0$, depending only on $G$, such that the following
holds, with the additional assumption
$L_k\geq T(1+\log k+|\log\vartheta|)$ when
$k\geq2$.
Define $q_0:=1$ and, for $0\leq i\leq k-2$,
\[
 q_{i+1}:=q_i-\min\{k^{-1},\vartheta q_i/2\}.
\]
There are an integer $0\leq j\leq k-1$ and a decomposition
\begin{equation}\label{eq:decomposition}
 \nu=\lambda\nu_{\proj}+q\nu_{\inv},\qquad
 q:=q_j,\qquad\lambda:=1-q,
\end{equation}
into probability measures, with zero terms omitted, for which
$\nu_{\inv}$ is $\eps_j$-almost stationary. If $\lambda>0$, then $\nu_{\proj}$ is
$(R_k,\delta_k)$-projective, where one may take
\begin{equation}\label{eq:product-rate}
 R_k:=U_k^a,\qquad
 \delta_k:=C_1 L_k^cU_k^{-a-c/2},\qquad
 R_k\delta_k=C_1\frac{L_k^c}{U_k^{c/2}}.
\end{equation}
If $j<k-1$, then $\nu_{\inv}$ is
\[
 \left(f,C\left(|\log\eps_j|^{-c}
                  +\min\left\{\vartheta,\frac{2}{kq}\right\}^{1/4}\right)\right)
 \text{-almost invariant}\qquad(f\in\Ff).
\]
If $j=k-1$, then for every $\eta\in[0,1]$ either $\nu_{\inv}$ is
$(f,C(|\log\eps_j|^{-c}+\eta^{1/4}))$-almost invariant for every
$f\in\Ff$, or, for every $\zeta>|\log\eps_j|^{-1/40}$, there is
a set of $\nu_{\inv}$-measure at least $\eta$ with an
$(R_\zeta,\delta_\zeta)$-projective factor, where
\[
 R_\zeta:=|\log\eps_j|^{\zeta/(40\dim G)},\qquad
 \delta_\zeta:=C|\log\eps_j|^{c'\zeta/(40\dim G)}\zeta^c.
\]
Moreover,
\[
 q_j\geq(1-\vartheta/2)^j,\qquad
 q_{k-1}\leq\min\{1,2/(\vartheta k)\}.
\]
For every $k\geq2$ and every $f\in\Ff$, the finite measure
$q\nu_{\inv}$ is
\begin{equation}\label{eq:weighted-invariance}
 (f,C_1(L_k^{-b}+(\vartheta k)^{-1/4}))\text{-almost invariant},
\end{equation}
and its normalisation $\nu_{\inv}$ is
\begin{equation}\label{eq:normalized-invariance}
 (f,C_1(L_k^{-b}+(\vartheta kq)^{-1/4}))\text{-almost invariant}.
\end{equation}
\end{theorem}

For $k=1$ and $\Ff=\{f\}$ this is exactly
Theorem~\ref{thm: qnz intro}, with both parameters $\eta$ and $\zeta$ retained.

\subsection{Statements of the key lemmas}
The two lemmas below handle the invariant and projective alternatives.
\begin{lemma}[Almost invariant measure on $G/\Gamma$]\label{lem: almost invariance}
Assume that $G$ has property~$(T)$. There is $\eps_{\mathrm T}>0$, depending only on $G$, with the following property.
Let $\nu$ be a probability measure on $G/\Gamma$ having a density
which is positive almost everywhere relative to $m_{G/\Gamma}$.
For $g\in G_1$, define $f_g:=d\nu/d(g\nu)$.
If, for some $0\leq\eps<\eps_{\mathrm T}$,
$W_1^{f_g}(\nu,g\nu)\leq\eps$ for every $g\in G_1$, then
$\Gamma$ is a lattice. If $C\subseteq G/\Gamma$ is compact and
$\nu(C)\geq1/2$, then
\[
 \cov(G/\Gamma)\leq4m_{G/\Gamma}(C).
\]
\end{lemma}

\begin{lemma}[Almost factor for a point in the support]\label{lem: almost factor}
Assume that $G$ has trivial centre. There are $r_0,c_*,C_*,\beta>0$ and an integer $J_G\geq1$, depending
only on $G$, with the following property. Let $x$ be a point in a
homogeneous space with discrete stabiliser $\Gamma_x:=\Stab_G(x)$.
Suppose there is a proper parabolic subgroup $Q\leq G$ such that
\[
 d_H^+(\Gamma_x\cap G_R^{\|},Q)\leq\delta.
\]
Let $\eta'\in(0,1)$ and $0<r\leq r_0/2$ satisfy
\[
 \bigl(A(\Ii(\Gamma_x))r^\varkappa\bigr)^{1/2}
       \leq\frac{\eta'}{32J_G},\qquad
 \delta R<c_*r,\qquad
 \log R\geq
       \left(\eta'^{-2}r^{-\dim G}\log\frac1r\right)^9.
\]
For
\[
 m:=\floor{\frac{\log R-\tfrac18\log\log R}{C_*}},
\]
one then has
\[
 (A_\mu^m\delta_x)
 \{y:\Stab_G(y)\cap G_{(\log R)^\beta}^{\|}=\{e\}\}
 \geq1-\eta'.
\]
\end{lemma}

 We also use the following consequence of
the Margulis inequality.
\begin{lemma}[Uniform discreteness]\label{lem: margulis functions intro}
There are $\varkappa>0$, $0<c_0<1$ and $b_0>0$, depending only on $G$, such that,
with
\[
 A(\iota):=\iota^{-\varkappa}+\frac{b_0}{1-c_0},\qquad \infty^{-\varkappa}:=0,
\]
every discrete $\Lambda\leq G$ satisfies
\[
 (A_\mu^n\delta_\Lambda)\{\Gamma:\Ii(\Gamma)\leq u\}
 \leq A(\Ii(\Lambda))u^\varkappa\qquad(n\geq1,\ u>0).
\]
\end{lemma}
The proof is given in Lemma~\ref{lem: margulis functions}; the transfer
from the capped logarithmic radius to $\Ii$ is proved in Appendix~\ref{app: margulis}.

\subsection{Proof of Theorem \ref{thm:genesis}}
\begin{proof}[Proof of Theorem~\ref{thm:genesis}]
We fix the quantities $L_k=\log^{(3k)}(n/2)$,
$U_k=e^{L_k}$ and $S_k=\log L_k$ as in Theorem \ref{thm:ergodic-qnz}.

We first choose the constants. Let $C_{\mathrm{NZ}}\geq1$ bound
the almost-invariance constants in Theorems~\ref{thm: qnz intro}
and~\ref{thm:ergodic-qnz}. Put $A_\Lambda:=A(\Ii(\Lambda))$. Let $c_2,L_G,D_G$ be the ratio, Lipschitz and error constants in
Claim~\ref{cl:regularization}, and let $C_H,c_H$ be the concentration
constants in Lemma~\ref{lem:concentration}. Put
$B_G:=\max\{1,c_2,L_G\}$. Define
\[
 \begin{split}
 M_*&:=\ceil{\max\left\{16,4B_GD_G/\eps_{\mathrm T}\right\}},\\
 \vartheta_G&:=\min\{1/4,(4C_{\mathrm{NZ}}M_*)^{-4}\},\qquad
 \eta_0:=\vartheta_G/8,\qquad p_G:=\vartheta_G/2,\\
 B_\Lambda&:=\max\left\{2/r_0,e,
       \bigl(3(64J_G)^2A_\Lambda\eta_0^{-3}\bigr)^{1/\varkappa}
       \right\},\\
 \gamma&:=\frac{3}{8\varkappa},\qquad
 p:=\frac94+9(\dim G+1)\gamma,\qquad
 P:=\max\{2,p,\gamma\}.
 \end{split}
\]
Decrease the exponent $a$ in Theorem~\ref{thm:ergodic-qnz}, if necessary,
so that $0<a<c/(2(1+c'))$.
Fix $T_G\geq T(2+|\log\vartheta_G|)$ so that, for every $t\geq T_G$,
\[
 \begin{gathered}
 C_1t^ce^{-ct/2}\leq t^{-1},\qquad
 (a/2)t-\tfrac18\log((a/2)t)\geq C_*,\\
 t\geq(4C_{\mathrm{NZ}}M_*)^{1/c},\qquad C_Ht e^{-c_Ht}\leq1/8,\\
 2C_m(\log t)^{9/8}/t\leq\eta_0,
 \qquad C_m:=\max\{1,a/(2C_*)\}.
 \end{gathered}
\]
Take $T_G$ larger than the fixed thresholds for the parameter choice in Theorem \ref{thm:ergodic-qnz} and the ratio between volumes of norm-ball and Riemannian ball. These constants depend only on $G$.
Set
\[
 D_\Lambda:=\max\left\{2T_G,
       \frac{2\eta_0^{-18}}{a}B_\Lambda^{9(\dim G+1)},
       \frac{2B_\Lambda}{c_*}\right\}.
\]
These choices give the asserted bound on $D_\Lambda$, after increasing $C_G$.

Set $\nu:=\nu_n$ and $\eps:=2/n$. By \eqref{eq:averaging}, $\nu$
is $\eps$-almost stationary. And we have 
\begin{equation}\label{eq:genesis-scales}
 \begin{aligned}
 L_k=\log^{(3k)}(1/\eps),\qquad
 U_k=\log^{(3k-1)}(1/\eps)=e^{L_k},\\
 S_k=\log L_k=\log^{(3k+1)}(1/\eps).
 \end{aligned}
\end{equation}
Apply Theorem~\ref{thm:ergodic-qnz} with
$\vartheta:=\vartheta_G$ and $\Ff:=\Ff_1(G/\Lambda)$. This yields a decomposition of $\nu$ 
as $\nu=\lambda\nu_{\proj}+q\tau$, and an integer $1\leq j \leq k-1 $ be the
stopping index. We have also
\[
 q\geq(1-\vartheta_G/2)^j\geq2^{-j},\qquad
 |\log\eps_j|\geq L_k.
\]

The proof will now proceed into to cases: (i) either the invariant part of $\nu$ is large enough, or (ii) $\nu$ is itself almost projective. These two cases can be detected by the parameters $j$ and $\eta$ from Theorem \ref{thm:ergodic-qnz}.

To deal with each case, we simply apply the two preliminary lemmas. Namely, in case (i) we apply Lemma \ref{lem: almost invariance} while in case (ii) we apply Lemma \ref{lem: almost factor}. 

Note that if we did not care about the $\log^{(4)}$ growth in Corollary \ref{cor: gap} the proof would be only a few more lines. Optimizing for this bound requires to deal carefully with the parameters at play. Let us do that now.

Let us start with case (i). If $j<k-1$, the measure $\tau$ is weakly
$C_{\mathrm{NZ}}(L_k^{-c}+\vartheta_G^{1/4})$-almost invariant. If $j=k-1$,
take the parameter $\eta$ in the statement of Theorem \ref{thm:ergodic-qnz} to be $\vartheta_G$.
The invariant alternative of Theorem \ref{thm:ergodic-qnz} also yields that $\tau$ is weakly
$C_{\mathrm{NZ}}(L_k^{-c}+\vartheta_G^{1/4})$-almost invariant. By the choices of
$\vartheta_G$ and $T_G$, the almost invariance of $\tau$ is at most $M_*^{-1}$.

In either of these two invariant cases, apply Claim~\ref{cl:regularization} with
$\varepsilon:=M_*^{-1}$. It gives $\widetilde\tau$ satisfying
\eqref{eq:localized-regularization}, with a positive density
$\widetilde\rho$ relative to $m_{G/\Lambda}$. For $g\in G_1$, define
\[
 \widetilde f_g:=\frac{d\widetilde\tau}{d(g\widetilde\tau)}
       =\frac{\widetilde\rho}{\widetilde\rho\circ g^{-1}}.
\]
The Lipschitzity bounds in Claim~\ref{cl:regularization} for the ratio $\tilde{f}_g$ and the weak almost invariance of $\tau$ imply 
\[
 W_1^{\widetilde f_g}(\widetilde\tau,g\widetilde\tau)
 \leq B_G W_1^b(\widetilde\tau,g\widetilde\tau)
 \leq B_GD_GM_*^{-1}<\eps_{\mathrm T}.
\]
The lattice criterion (Lemma~\ref{lem: almost invariance}) shows that
$\Lambda$ is a lattice.

To bound the covolume, use \(q\tau\leq\nu\) and the fixed radius
\(s_G\) in Claim~\ref{cl:regularization}. Set
\(C:=\overline G_{v_Gn+s_G}o\). The concentration estimate and the
localisation assertion of Claim~\ref{cl:regularization} and Lemma \ref{lem:concentration} give
\[
 \widetilde\tau(C)\geq1-C_Gq^{-1}e^{-c_Gn}-\tfrac18.
\]
The conditions on $k,n, T_G$ now imply \(k\leq\log n\) and \(n\geq T_G\).
Consequently \(q^{-1}\leq2^{k-1}\leq n\), and
\[
 \widetilde\tau(C)\geq1-C_Gn e^{-c_Gn}-\tfrac18\geq\tfrac12.
\]
Lemma~\ref{lem: almost invariance} now gives
\[
 \cov(G/\Lambda)\leq4m_{G/\Lambda}(C)
       \leq4m_G(\overline G_{v_Gn+s_G}).
\]
Since \(s_G\) depends only on \(G\), the exponential volume bound proves
the first alternative.

It remains to consider the projective alternative (i.e. (ii) in the proof strategy outlined earlier). Here
$j=k-1$, and there is a set $W$ of $\tau$-measure at least
$\vartheta_G$ having a projective factor. Choose its QNZ parameter
by setting $\zeta:=40(\dim G)\alpha(\log u)/u$, with
$u:=\log|\log\eps_{k-1}|$ and $\alpha:=c/(2(1+c'))$.
For $k=1$ one has $u=U_1$, and for $k\geq2$ one has
$u=U_k-\log2$. Substitution into Theorem~\ref{thm: qnz intro}, with the fixed
choice $a<\alpha$, shows that $W$ has the common parameters
\[
 R_k:=U_k^a,\qquad
 \delta_k:=C_1L_k^cU_k^{-a-c/2}.
\]
The extracted part has the same parameters by
Theorem~\ref{thm:ergodic-qnz}. Let $B$ be the union of $W$ and a
measurable set on which $\nu_{\proj}$ is supported with these
parameters, omitting the latter set when $\lambda=0$. Since
$\lambda\nu_{\proj}+q\tau|_W\leq\nu$, we obtain
\begin{equation}\label{eq:projective-mass}
 \nu(B)\geq1-q+q\vartheta_G
 \geq b_k,\qquad
 b_k:=\max\{\vartheta_G,1-2/(\vartheta_Gk)\}.
\end{equation}

Put
\begin{gather*}
 \eta:=\eta_0k^{-1/8},\qquad
 u:=\left(\frac{\eta}{2A_\Lambda}\right)^{1/\varkappa},\\
 \Gamma_x:=\Stab_G(x),\qquad
 X_0:=B\cap\{x:\Ii(\Gamma_x)>u\}.
\end{gather*}
Uniform discreteness gives
\[
 \nu(X_0)\geq b_k-\eta/2,
\]
and, for every $x\in X_0$,
\[
 A(\Ii(\Gamma_x))\leq H_\Lambda(\eta),\qquad
 H_\Lambda(\eta):=\frac{2A_\Lambda}{\eta}+\frac{b_0}{1-c_0}
       \leq3A_\Lambda\eta^{-1}.
\]
Define
\begin{equation}\label{eq:r-choice}
 r:=\min\left\{r_0/2,e^{-1},
       \left(\frac{\eta^2}{(64J_G)^2H_\Lambda(\eta)}\right)^{1/\varkappa}
       \right\}.
\end{equation}
Then, uniformly for $x\in X_0$,
\[
 \bigl(A(\Ii(\Gamma_x))r^\varkappa\bigr)^{1/2}
       \leq\frac{\eta}{64J_G},\qquad
 r\geq B_\Lambda^{-1}k^{-\gamma}.
\]
Take $R:=U_k^{a/2}\leq R_k$, so that $\log R=(a/2)L_k$.
Since $\log(1/r)\leq r^{-1}$, the scale condition and the definition
of $D_\Lambda$ imply
\[
 \left(\eta^{-2}r^{-\dim G}\log\frac1r\right)^9
 \leq\eta_0^{-18}B_\Lambda^{9(\dim G+1)}k^p
 \leq(a/2)L_k=\log R.
\]
Using $U_k=e^{L_k}$ from \eqref{eq:genesis-scales}, we also have
\[
 \delta_kR\leq\delta_kR_k
 \leq C_1L_k^ce^{-cL_k/2}
 \leq L_k^{-1}
 \leq\frac{c_*}{2B_\Lambda}k^{-\gamma}<c_*r.
\]
Thus Lemma~\ref{lem: almost factor} applies with $\eta':=\eta$ and
these parameters for every $x\in X_0$.

Define
\[
 m:=\floor{\frac{\log R-\tfrac18\log\log R}{C_*}},\qquad
 F_R:=\{y\in G/\Lambda:
       \Stab_G(y)\cap G_{(\log R)^\beta}^{\|}=\{e\}\}.
\]
The parameters are the same for every point in $X_0$, and
$1\leq m\leq C_mL_k$. Integrating the almost-factor lemma and then
using \eqref{eq:averaging} yields
\begin{equation}\label{eq:finite-error}
 \begin{aligned}
 \nu(F_R)&\geq(1-\eta)\nu(X_0)-2m/n\\
          &\geq b_k-\tfrac32\eta-2m/n
           \geq b_k-\tfrac52\eta.
 \end{aligned}
\end{equation}
For the last inequality, use $L_k\leq\log n$, $k\leq\log n$ and
$2C_m(\log n)^{9/8}/n\leq\eta_0$ to obtain $2m/n\leq\eta$.
Since $b_k\geq\vartheta_G$ and
$\eta\leq\vartheta_G/8$, the last expression in
\eqref{eq:finite-error} is at least $p_G=\vartheta_G/2$.
It is also at least
\[
 1-\frac{2}{\vartheta_Gk}-\frac52\eta_0k^{-1/8}
 \geq1-C_Gk^{-1/8}.
\]
This proves \eqref{eq:genesis-conclusion}, including $k=1$.

The norm-ball comparison $G_{c\log T}\subseteq G_T^{\|}$ gives
the stated metric radius, since
$\log((\log R)^\beta)=\beta S_k+\beta\log(a/2)$.
In particular $S_1=\log^{(4)}(n/2)$.
Finally, if $N:=\log^*(n/2)$ and $k:=\floor{N/4}$, then
$L_k>\exp^{(N-3k-1)}(e)\geq\exp^{(k-1)}(e)$.
Hence $L_k/k^P\to\infty$ as $n\to\infty$, proving the final
assertion for every fixed $\Lambda$.
\end{proof}

\begin{proof}[Deduction of Theorem~\ref{thm A}]
Choose a lift \(g_0\in G\) of \(o\), put
\(\Lambda:=g_0\Gamma g_0^{-1}\), and identify \(G/\Lambda\) with
\(G/\Gamma\) by \(h\Lambda\mapsto hg_0\Gamma\).
The displacement formula in the proof of
Lemma~\ref{lem: geometric comparisons} gives
\(\Ii(\Lambda)\geq2\operatorname{inj}(o)\). Thus, with
\(d=\dim G\), Theorem~\ref{thm:genesis} gives
\[
 D_\Lambda\leq C_G\iota^{-9(d+1)}.
\]
Take \(a_G:=9(d+1)\) and, initially, set
\[
 k:=\max\{1,\lceil(C_G/\varepsilon)^8\rceil\},\qquad m:=3k+1,
\]
using the error constant of Theorem~\ref{thm:genesis}.
Choose \(C_{G,\varepsilon}\) large enough that the hypothesis of
Theorem~\ref{thm A} implies
\[
 L:=\log^{(3k)}(n/2)\geq D_\Lambda k^P,
 \qquad \log L\geq\max\{1,2D_K/c_G\},
\]
where \(c_G\) in this display is the metric-radius constant of
Theorem~\ref{thm:genesis}. This choice depends only on \(G,\varepsilon\).

Apply Theorem~\ref{thm:genesis}. Its first alternative gives the volume
bound by \eqref{eq: volume normalization}, after increasing \(C_G\).
In its second alternative, a proportion at least \(1-\varepsilon\)
has discreteness radius at least \(c_G\log L\). Equation
\eqref{eq: symmetric injectivity comparison} therefore gives injectivity
radius at least \(c_G\log L/4\).
All intervening logarithms of \(n/2\) are at least \(1\), so successive
applications of \(\log\) do not increase the initial difference
\(\log n-\log(n/2)=\log2\). Since \(\log L\geq1\), it follows that
\[
 \log^{(m)}n\leq\log^{(m)}(n/2)+\log2
             \leq2\log^{(m)}(n/2).
\]
Decreasing the radius constant proves the second alternative of
Theorem~\ref{thm A}. If \(\varepsilon\geq1-p_G\), take \(k=1\)
throughout and use the mass bound \(p_G\) instead; this gives \(m=4\).

For a singular basepoint, use
\(\iota:=\min\{1,\Ii(\Lambda)\}\) directly in the bound for
\(D_\Lambda\). This is positive because \(\Lambda\) is discrete.
Changing the lift conjugates \(\Lambda\) by an element of \(K\), so
the value of \(\iota\) is unchanged. The rest of the argument is the same.
\end{proof}

\section{Almost invariance and property \texorpdfstring{$(T)$}{(T)}}
\label{sec: almost invariance}

We prove the lattice criteria (Theorem \ref{thm: B} and Corolary \ref{cor: weak lattice criterion, intro}), as well as the the quantitative estimate needed
for the covolume conclusion in Theorem~\ref{thm:genesis}. The main engine is property $(T)$ and a well chosen oservable. 

\begin{lemma}[Almost invariant measure on $G/\Gamma$]
\label{lem: almost invariance proof}
Assume that $G$ has property~$(T)$, and choose a Kazhdan constant
$\kappa_G>0$ for $G_1$. Put $\eps_{\mathrm T}:=\kappa_G^2/32$.
Let $\nu=\rho m_{G/\Gamma}$ be a probability measure with $\rho>0$
almost everywhere, and set $f_g:=\rho/(\rho\circ g^{-1})$.
If $0\leq\eps<\eps_{\mathrm T}$ and
\[
 W_1^{f_g}(\nu,g\nu)\leq\eps\qquad(g\in G_1),
\]
then $\Gamma$ is a lattice. Writing $V:=m_{G/\Gamma}(G/\Gamma)$, one has
\begin{equation}\label{eq: distance to Haar}
 d_{\mathrm{TV}}(\nu,V^{-1}m_{G/\Gamma})
 \leq\frac{\sqrt{2\eps}}{\kappa_G}<\frac14.
\end{equation}
Consequently, if $C\subseteq G/\Gamma$ is compact and $\nu(C)\geq1/2$,
then $V\leq4m_{G/\Gamma}(C)$.
\end{lemma}

\begin{proof}
The generating neighbourhood $G_1$ is a Kazhdan set by
\cite[Proposition~1.3.2]{BHV}. Write $m:=m_{G/\Gamma}$ and let
$\pi$ be the unitary representation on $L^2(G/\Gamma)$ given by
$(\pi(g)\xi)(x):=\xi(g^{-1}x)$. The function $\xi:=\sqrt\rho$ is a unit vector.

Fix $g\in G_1$ and put $q(x):=\rho(g^{-1}x)$. Positivity ensures that
$f_g=\rho/q$ is defined almost everywhere for both $\nu$ and $g\nu$.
The function
\[
 \phi(t):=\frac{t-1}{t+1}\quad(t>0),\qquad \phi(t):=-1\quad(t\leq0),
\]
is bounded by $1$ and is $2$-Lipschitz. Applying the hypothesis to $\phi/2$ gives
\begin{align*}
 \int\frac{(\rho-q)^2}{\rho+q}\,dm
 &=\int\phi(f_g)(\rho-q)\,dm\\
 &=\int\phi(f_g)\,d\nu-\int\phi(f_g)\,d(g\nu)
 \leq2\eps.
\end{align*}
Since
\[
 (\sqrt\rho-\sqrt q)^2
 =\frac{(\rho-q)^2}{(\sqrt\rho+\sqrt q)^2}
 \leq\frac{(\rho-q)^2}{\rho+q},
\]
we obtain $\|\xi-\pi(g)\xi\|_2\leq\sqrt{2\eps}<\kappa_G$ for every
$g\in G_1$. Property~$(T)$ gives a nonzero invariant vector.
Its squared absolute value defines a nonzero finite invariant measure
on $G/\Gamma$, so $\Gamma$ is a lattice.

Let $P_{\mathrm{inv}}$ be the orthogonal projection onto the invariant
vectors. Applying the Kazhdan inequality on their orthogonal complement yields
\[
 \|\xi-P_{\mathrm{inv}}\xi\|_2
 \leq\frac{\sqrt{2\eps}}{\kappa_G}.
\]
Now, the invariant vectors are the constants. Set $\zeta:=V^{-1/2}$ and
$a:=\langle\xi,\zeta\rangle\geq0$. Cauchy--Schwarz gives
\begin{align*}
 d_{\mathrm{TV}}(\nu,V^{-1}m)
 &=\frac12\int|\xi^2-\zeta^2|\,dm\\
 &\leq\frac12\|\xi-\zeta\|_2\|\xi+\zeta\|_2
 =\sqrt{1-a^2}=\|\xi-P_{\mathrm{inv}}\xi\|_2.
\end{align*}
This proves \eqref{eq: distance to Haar}. If $\nu(C)\geq1/2$, it follows that
$m(C)/V\geq1/4$, which is the covolume bound.
\end{proof}

Although not useful in the proof of our main result, the following corollary appears to be of independent interest.

\begin{corollary}\label{cor: weak lattice criterion}
Assume that $G$ has property~$(T)$. There is $\eps_{\mathrm w}=\eps_{\mathrm w}(G)>0$ such that, for every
discrete $\Gamma\leq G$, the existence of a probability measure $\tau$ on
$G/\Gamma$ satisfying
\[
 W_1^b(\tau,g\tau)\leq\eps_{\mathrm w}\qquad(g\in G_1)
\]
implies that $\Gamma$ is a lattice.
\end{corollary}

\begin{proof}
Take $M=M(G)$ sufficiently large and apply
Claim~\ref{cl:regularization} with $\eps_{\mathrm w}:=M^{-1}$.
The resulting probability $\widetilde\tau$ has a positive density, and
its ratios $\widetilde f_g$ are bounded by $c_2$ and are $L_G$-Lipschitz,
with constants depending only on $G$. Put $B_G:=\max\{1,c_2,L_G\}$.
Lemma~\ref{lem: almost invariance proof} applies, since
\[
 W_1^{\widetilde f_g}(\widetilde\tau,g\widetilde\tau)
 \leq B_G W_1^b(\widetilde\tau,g\widetilde\tau)
 \leq B_G D_G M^{-1}<\eps_{\mathrm T}.\qedhere
\]
\end{proof}

\section{The almost factor case}\label{sec: almost factor}
Throughout this section $G$ is centre-free. The goal is to prove
Lemma~\ref{lem: almost factor}.
The proof relies on choosing a number of geometric constants and parameters.
We record the ones used later in the argument.

\begin{nota}[Constants fixed once and for all]\label{nota: fixed constants}
The following constants depend only on $G$.
\begin{enumerate}
 \item[(C1)] $C_0\geq1$, $\delta_0>0$ and $b\geq1$ are the constants
 of Lemma~\ref{lem: quantitative discreteness}, made uniform over the
 finitely many standard maximal parabolics and the Levi factors used below.
 \item[(C2)] $J_G\leq\dim G$ and $D\geq1$ are the constants of
 Lemma~\ref{lem: trivial intersection}. The integer $n_1$ bounds the
 number of unit steps needed for any of the finitely many contraction
 elements used in Lemma~\ref{lem: factor final}.
 \item[(C3)] $r_0\leq1$ is a sufficiently small radius, fixed below, and
 \[
 c_*:=\frac{1}{4C_0D},\qquad \beta:=\frac{1}{8b}.
 \]
 \item[(C4)] The constants $\varkappa,c_0,b_0$ are those of
 Lemma~\ref{lem: radius transfer}, and
 \[
 A(\iota):=\iota^{-\varkappa}+\frac{b_0}{1-c_0}.
 \]
 \item[(C5)] $C_*,A_G,C_G$ are the constants of
 Lemma~\ref{lem: upgraded furstenberg}, and $T_m:=e^{C_*m}$.
\end{enumerate}
\end{nota}

The general idea of the proof is to use the almost stationarity of the average
$A_\mu^m\delta_\Gamma$ to obtain almost invariance along a Borel subgroups of the measures in
its almost Furstenberg decomposition \cite{QNZ}. We then use this almost invariance to
convert proximity of the stabiliser to a parabolic subgroup to proximity to the trivial
subgroup. As in the qualitative approach from \cite{FG}, there are two steps. First we pass from proximity to the parabolic to proximity to the Levi factors. We then use trivial intersection of the Levi factors to conclude.

Both steps are inspired by the method of \cite{FG}. Their quantitative
forms require the polynomial estimate for the Furstenberg decomposition,
the measurable pigeonhole argument in Claim~\ref{cl: prob pigeon hole},
and the estimate for the intersection of Levi factors in
Lemma~\ref{lem: trivial intersection}.

\subsection{From parabolic to Levi}
Our first intermediate result, Lemma~\ref{lem: quantitative discreteness} below, turns proximity to a parabolic
subgroup into proximity to any fixed Levi factor, with a controlled loss
of probability. Recall that $\Gamma$ is $r$-discrete when $\Ii(\Gamma)\geq r$.

\begin{nota}
For $U\subset G$, write
$\Gamma[U]:=\{\Lambda\in\subd:\Lambda\cap U\neq\emptyset\}$
for the event that a discrete subgroup meets $U$.
\end{nota}

We prove the following quantitative version of \cite[Lemma 6.2]{FG}.

\begin{lemma}\label{lem: quantitative discreteness}
Let $Q=LN$ be a Levi decomposition of a proper parabolic subgroup
of the centre-free group $G$, and let $A_L$ be the centre of $L$.
There are $C_0\geq1$, $\delta_0>0$, $b\geq1$, $R_2$ and
$n_1\in\NN$, depending on the fixed $G,Q,L$ and the norms, such that
the following holds. Let $R\geq R_2$, $0<r\leq1$,
$0<\delta'\leq\delta_0$, $C_0\delta'<r$, $0<c<1$ and $\delta,p\geq0$. Suppose
\begin{equation}\label{eq: quantitative discreteness parameters}
 p<c/8,\qquad
 n_1\delta+p<\frac{c/16}
 {(8R^b r^{-\dim G}/c+1)\log_2(R^b/\delta')}.
\end{equation}
Put $R^*:=R^b$. If $\lambda$ is a probability on $\subd$,
norm-$\delta$-almost invariant under $A_L$, satisfying
\begin{equation}\label{eq: assumption}
 \lambda\{\Gamma:d_H^+(\Gamma_{R^*}^{\|},Q)\leq\delta'\}\geq1-p,
 \qquad
 \lambda\{\Gamma:\Ii(\Gamma)\geq r\}\geq1-c/8,
\end{equation}
then
\[
 \lambda\{\Gamma:d_H^+(\Gamma_R^{\|},L)\leq2C_0\delta'\}\geq1-c.
\]
\end{lemma}

\begin{proof}
Choose $H_0$ in the Lie algebra of the split centre of $L$ so that
$a_0:=\exp(H_0)$ contracts $\mathfrak n:=\Lie(N)$. The action of
$\Ad(a_0)$ on $\mathfrak n$ is diagonalizable with positive
eigenvalues. After multiplying $H_0$ by a fixed positive constant,
choose a Euclidean norm $\|\cdot\|_N$ on $\mathfrak n$ such that
\begin{equation}\label{eq: logarithmic contraction}
 \|Y\|\leq\|Y\|_N,\qquad
 \|\Ad(a_0)Y\|_N\leq\tfrac12\|Y\|_N.
\end{equation}
Here $\|\cdot\|$ is the Cartan norm fixed in Section~\ref{sec: def}.
Define the closed logarithmic balls
\[
 N_s^{\log}:=\{\exp Y:Y\in\mathfrak n,\ \|Y\|_N\leq s\}.
\]
The exponential map is a diffeomorphism from $\mathfrak n$ to $N$.
In particular,
\begin{equation}\label{eq: logarithmic balls}
 N_s^{\log}\subseteq\overline G_s,\qquad
 a_0^jN_s^{\log}a_0^{-j}\subseteq N_{2^{-j}s}^{\log}
 \quad(j\geq0).
\end{equation}
Fix $n_1$ larger than the length of the path
$t\mapsto\exp(tH_0)$, $0\leq t\leq1$, rounded up. Dividing this
path into $n_1$ pieces gives
\[
 |\lambda(a_0 E)-\lambda(E)|\leq n_1\delta
\]
for every measurable event $E$.

Let $V_Q$ be the unipotent radical opposite to $Q$ for this Levi
decomposition. The product map $V_Q\times Q\to G$ is injective
with open image. Its local inverse at $e$ gives $C_0\geq1$ and
$\delta_0>0$ such that the following claim holds. We also take
$C_0$ at least the Lipschitz constant of conjugation by $a_0$.

\begin{claim}[Regularity of the $V_QQ$-decomposition]
\label{cl: exact decomposition}
If $0<\varepsilon\leq\delta_0$ and
$d_G(g,Q)\leq\varepsilon$, then the unique decomposition $g=vq$
has $v\in V_Q\cap\overline G_{C_0\varepsilon}$.
\end{claim}
\begin{proof}
The metric is proper and $Q$ is closed, so choose $q'\in Q$ with
$d_G(g,q')\leq\varepsilon$. Right invariance gives
$g(q')^{-1}\in\overline G_\varepsilon$. Apply the local product
chart to this element, and then multiply its $Q$-component by $q'$.
Uniqueness proves the assertion.
\end{proof}

We next record the size of the logarithmic coordinates used below.
There are an exponent $a\geq1$ and a sufficiently large $R_2$ such
that, after increasing $C_0$, if $\gamma\in G_R^{\|}$ and
$\gamma=vq$ with $v\in V_Q\cap\overline G_1$, then
\begin{equation}\label{eq: parabolic logarithmic coordinates}
 q=\exp(Y)\ell,\qquad
 \|Y\|_N\leq R^a,\qquad d_G(e,\ell)\leq C_0\log R.
\end{equation}
To see this, use the eigenspace decomposition of $\operatorname{ad}(H_0)$
on $\mathfrak g$. The matrices of $\Ad(q)$ are block triangular,
and their diagonal blocks are those of $\Ad(\ell)$; the same is
true of $q^{-1}$. Projection onto the diagonal blocks is a fixed
bounded linear map, so $|\ell|_{\Ad}\leq C|q|_{\Ad}$ and
$|\exp Y|_{\Ad}\leq C|q|_{\Ad}^2$.
Also $|q|_{\Ad}\leq C R$, since $v\in\overline G_1$.
Writing $d:=\dim G$, the nilpotence of $\operatorname{ad}(Y)$ gives
\[
 \operatorname{ad}(Y)
 =\log\Ad(\exp Y)
 =\sum_{j=1}^{d-1}\frac{(-1)^{j+1}}j
                  (\Ad(\exp Y)-I)^j.
\]
The map $\operatorname{ad}$ is injective. This polynomial identity
bounds $\|Y\|_N$ by a fixed power of $R$. The bound on $\ell$
follows from \eqref{eq: norm metric comparison}. This proves
\eqref{eq: parabolic logarithmic coordinates}.

For later use, the polynomial identity
$\Ad(\exp Y)=\sum_{j=0}^{d-1}\operatorname{ad}(Y)^j/j!$, and the
same identity for $-Y$, imply
\[
 d_G(e,\exp Y)\leq C\log(2+\|Y\|_N).
\]
Thus, with $\Sigma:=C_0\log R$, there is a fixed $C_1$ such that
\begin{equation}\label{eq: logarithmic tube}
 \overline G_1N_{R^a}^{\log}\overline L_\Sigma
       \subseteq G_{C_1\log R},
 \qquad \overline L_\Sigma:=L\cap\overline G_\Sigma.
\end{equation}
These estimates are fixed before choosing $b$. Choose $b\geq a$
large enough that the last ball is contained in $G_{R^b}^{\|}$,
and that the number of points of any $r$-discrete subgroup in that
ball is strictly smaller than $R^b r^{-d}$ for $0<r\leq1$.
The latter follows from the packing estimate proved below. Fixed
multiplicative constants are absorbed by increasing $R_2$.
Finally, shrink $\delta_0$ so that $C_0^2\delta_0\leq1$.

Suppose, towards a contradiction, that the conclusion fails.
The two assumptions on $\lambda$ give
\begin{equation}\label{eq: far far away}
 \lambda\{\Gamma:d_H^+(\Gamma_R^{\|},L)>2C_0\delta',\
                 d_H^+(\Gamma_{R^*}^{\|},Q)\leq\delta'\}
 \geq c-p\geq c/2.
\end{equation}
Define
\[
 U_0:=(N_{R^a}^{\log}\setminus N_{C_0\delta'}^{\log})
           \overline L_\Sigma,\qquad
 U_j:=a_0^jU_0a_0^{-j}\quad(j\geq0).
\]
Since $a_0$ centralizes $L$, \eqref{eq: logarithmic balls} gives
\begin{equation}\label{eq: Ui structure}
 U_j\subseteq N_{R^a}^{\log}\overline L_\Sigma\subset Q.
\end{equation}
For $s>0$, put $V(s):=V_Q\cap\overline G_s$.

\begin{claim}\label{cl: square brackets}
Every subgroup in the event of \eqref{eq: far far away} meets
$V(C_0\delta')U_0$. Consequently,
\begin{equation}\label{eq: lower bound on first U}
 \lambda(\Gamma[V(C_0\delta')U_0])\geq c/2.
\end{equation}
\end{claim}
\begin{proof}
For $\gamma\in\Gamma_R^{\|}$, the parabolic proximity and
Claim~\ref{cl: exact decomposition} give $\gamma=vq$ with
$v\in V(C_0\delta')$. By
\eqref{eq: parabolic logarithmic coordinates}, write
$q=\exp(Y)\ell$ with $\|Y\|_N\leq R^a$ and
$\ell\in\overline L_\Sigma$. If $\|Y\|_N\leq C_0\delta'$, then
right invariance and \eqref{eq: logarithmic balls} give
\[
 d_G(\gamma,L)\leq d_G(v\exp(Y)\ell,\ell)
 \leq d_G(e,v)+d_G(e,\exp Y)\leq2C_0\delta'.
\]
This cannot hold for every $\gamma\in\Gamma_R^{\|}$. Choosing one
for which it fails gives a witness in $V(C_0\delta')U_0$.
\end{proof}

\begin{claim}[Pigeonhole in probability spaces]\label{cl: prob pigeon hole}
If $M>8(\ell-1)/c$ and $A_1,\ldots,A_M$ have $\lambda$-measure
at least $c/8$, then some $\ell$ of them have nonempty intersection.
\end{claim}
\begin{proof}
Otherwise $\sum_{i=1}^M\ind_{A_i}\leq\ell-1$ everywhere, whereas
its integral is at least $Mc/8>\ell-1$.
\end{proof}

\begin{claim}[Disjoint contractions]\label{cl: disjoint contraction}
Put $s:=\lceil\log_2(R^b/\delta')\rceil$ and $p_i:=is$.
The sets $U_{p_i}$, $i\geq0$, are pairwise disjoint, and
\[
 U_{p_i}\subseteq N_{C_0\delta'}^{\log}\overline L_\Sigma
 \qquad(i\geq1).
\]
\end{claim}
\begin{proof}
Since $a\leq b$, \eqref{eq: logarithmic balls} gives
$2^{-s}R^a\leq\delta'\leq C_0\delta'$. This proves the inclusion.
Uniqueness of the $NL$-decomposition makes $U_0$ disjoint from
each $U_{p_i}$ with $i\geq1$. Conjugating by $a_0^{-p_i}$ then
shows that $U_{p_i}\cap U_{p_j}=\emptyset$ whenever $i<j$.
\end{proof}

\begin{claim}\label{cl: machado lemma}
For every $j\geq0$,
\[
 \lambda(\Gamma[V(C_0\delta')U_j])
       \geq c/2-j(n_1\delta+p).
\]
\end{claim}
\begin{proof}
The case $j=0$ is \eqref{eq: lower bound on first U}. Conjugation
by $a_0$ sends $V(C_0\delta')U_j$ into
$V(C_0^2\delta')U_{j+1}$, at a cost of at most $n_1\delta$ in
measure. Restricting to the parabolic-proximity event costs at most
$p$. Every witness $\gamma'=v'u'$ on this event lies in
$G_{R^*}^{\|}$ by \eqref{eq: logarithmic tube}. Thus
$d_G(\gamma',Q)\leq\delta'$, and
Claim~\ref{cl: exact decomposition} improves its unique
$V_Q$-component to $v'\in V(C_0\delta')$ without changing $u'$.
This proves the induction step.
\end{proof}

\begin{claim}[Pigeonhole in the group]\label{cl: group pigeon hole}
There are $C_2,c_2>0$, depending only on $G$, such that an
$r$-discrete subgroup, $0<r\leq1$, has at most
$C_2e^{c_2T}r^{-d}$ elements in $G_T$.
\end{claim}
\begin{proof}
Distinct subgroup elements have distance at least $r$, by right
invariance. Their open radius-$r/2$ balls are disjoint and lie in
$G_{T+1}$. Each has Haar measure at least $c r^d$, whereas
$m_G(G_{T+1})\leq C e^{c_2T}$. Dividing proves the claim.
\end{proof}

Put $\ell:=\lceil R^b r^{-d}\rceil+1$ and
$M:=\lfloor8(\ell-1)/c\rfloor+1$. The rounding gives
\[
 M\leq2(8R^b r^{-d}/c+1),\qquad
 s\leq2\log_2(R^b/\delta').
\]
The parameter condition therefore implies
$Ms(n_1\delta+p)<c/4$. For $1\leq i\leq M$,
Claim~\ref{cl: machado lemma} gives
$\lambda(\Gamma[V(C_0\delta')U_{p_i}])\geq c/4$.
Intersecting each event with the $r$-discrete event leaves measure
at least $c/8$. Claim~\ref{cl: prob pigeon hole} supplies an
$r$-discrete subgroup $\Gamma_0$ meeting at least $\ell$ of these
sets. Their $Q$-components are disjoint, so uniqueness of the
$V_QQ$-decomposition gives $\ell$ distinct subgroup elements in
$G_{C_1\log R}$. This contradicts the packing bound used to choose
$b$, since $\ell>R^b r^{-d}$.
\end{proof}

For a fixed finite family of Levi factors, choose the norms in
\eqref{eq: logarithmic contraction} separately, then choose $C_0$
uniformly before fixing the tube bound and $b$. All constants are
therefore uniform over this family. In particular, if a contraction element
is conjugated by $q_i\in Q$, one may use a rescaling of
$Y\mapsto\|\Ad(q_i^{-1})Y\|_N$ on $\mathfrak n$ and include the
length of the conjugated path in $n_1$.

\subsection{From Levi to the trivial subgroup}
We apply Lemma~\ref{lem: quantitative discreteness} to finitely many
Levi factors with trivial intersection. The next lemma constructs this
family and controls the distance to its intersection.

We first need the following slightly more detailed version of \cite[Lemma 6.3]{FG}.
The argument follows \cite{FG}, with a quantitative estimate for the intersection.

\begin{lemma}\label{lem: trivial intersection}
Let $G$ be a connected simple Lie group with trivial centre, equipped with a
right-invariant Riemannian metric $d_G$ fixed in Section~\ref{sec: def}.
Let $Q$ be a proper parabolic subgroup with Levi decomposition $Q=L\ltimes U$.
Then the intersection of all Levi subgroups of $Q$ is trivial. Moreover, there
exist $J_G\le\dim G$ and $q_1,\dots,q_{J_G}\in Q$ with
\[
  \bigcap_{i=1}^{J_G} q_iLq_i^{-1}=\{e\},
\]
and there are constants $D>0$, $\delta_0>0$, depending only on the fixed $G$, $d_G$, $Q$ and $L$,
such that for every $0<\delta\le\delta_0$ and every $g\in G$ with
$d_G(g,q_iLq_i^{-1})\le\delta$ for $i=1,\dots,J_G$, one has $g\in G_{D\delta}$.
\end{lemma}

\begin{proof}
Write $\mathfrak g,\mathfrak l,\mathfrak u$ for the Lie algebras of $G,L,U$, and
let $\mathfrak u^-$ be the nilradical of the opposite parabolic, so that
$\mathfrak g=\mathfrak u^-\oplus\mathfrak l\oplus\mathfrak u$ is the grading by the
eigenvalues of $\operatorname{ad}(H_0)$ for a suitable $H_0\in\mathfrak l$, with
$\mathfrak l$ the $0$-eigenspace. Since $Z(G)=\{e\}$, the map
$\operatorname{Ad}:G\to GL(\mathfrak g)$ is injective and identifies $G$ with the
identity component of the real points of the adjoint algebraic group.
The Levi subgroup $L$ is an open subgroup of finite index in the real
points of the corresponding reductive algebraic Levi subgroup. Use the Cartan norm fixed in Section~\ref{sec: def} and its
associated operator norm. Invariance of the Killing form $B$ under
$\operatorname{ad}(H_0)$ makes the pairing
$\mathfrak u\times\mathfrak u^-\to\mathbb R$, $(X,Y)\mapsto B(X,Y)$,
nondegenerate, as $B$ pairs the grading space $\mathfrak g_s$ with $\mathfrak g_{-s}$.

\smallskip
\emph{Step 1 (algebraic rigidity).} Let
$\psi:=\operatorname{Ad}(\cdot)|_{\mathfrak u}:L\to GL(\mathfrak u)$.
We claim $\psi$ is injective with injective differential
$\theta:=d\psi_e=\operatorname{ad}(\cdot)|_{\mathfrak u}$.
Indeed, the subalgebra generated by $\mathfrak u\cup\mathfrak u^-$ is a nonzero
ideal, hence all of $\mathfrak g$. If $\lambda\in L$ fixes $\mathfrak u$
pointwise, then for $Y\in\mathfrak u^-$, $X\in\mathfrak u$,
$B(\operatorname{Ad}(\lambda)Y-Y,X)=B(Y,\operatorname{Ad}(\lambda^{-1})X)-B(Y,X)=0$;
as $\operatorname{Ad}(\lambda)Y-Y\in\mathfrak u^-$ and the pairing is nondegenerate,
$\lambda$ also fixes $\mathfrak u^-$, hence all of $\mathfrak g$, so
$\lambda\in Z(G)=\{e\}$. The same computation with $\lambda$ replaced by
$\exp(tX)$, $X\in\mathfrak l$, differentiated at $t=0$, gives injectivity of
$\theta$.

\smallskip
\emph{Step 2 (properness and a local estimate).} The map $\psi$ is the restriction of an algebraic representation.
Its image is closed by \cite[Proposition~3.2]{BaderLeibtag}: that
proposition gives a topological quotient map on the real points, and
the image of their open finite-index subgroup $L$ is again open and
of finite index, hence closed. A bijective homomorphism of $\sigma$-compact
locally compact groups onto a Lie subgroup is an isomorphism onto it; thus $\psi$
is a closed embedding, and in particular proper. On the other hand, injectivity
of $\theta$ gives $c_0:=\min_{\|X\|=1}\|\theta(X)\|_{op}>0$, and
$\operatorname{Ad}(\exp X)|_{\mathfrak u}=e^{\theta(X)}$ yields $r_1\in(0,1]$ with
\begin{equation}\label{eq:local}
  \|X\|\le r_1\ \Longrightarrow\ \|\psi(\exp X)-I\|_{op}\ge\tfrac{c_0}{2}\|X\|.
\end{equation}
Let $W:=\exp\{X\in\mathfrak l:\|X\|<r_1\}$. Since
$F:=\{\lambda:\|\psi(\lambda)-I\|_{op}\le\tfrac12\}$ is compact (properness) and
$\|\psi(\cdot)-I\|_{op}$ vanishes on $F$ only at $e\in W$, the number
$\eta_0:=\tfrac12\min\bigl(\tfrac12,\ \inf_{F\setminus W}\|\psi(\cdot)-I\|_{op}\bigr)$ is
positive. If $\|\psi(\lambda)-I\|_{op}\le\eta_0$ then $\lambda\in W$, and
\eqref{eq:local} with $d_G(e,\exp X)\le\|X\|$ gives
\begin{equation}\label{eq:C_L}
  d_G(e,\lambda)\le C_L\,\|\psi(\lambda)-I\|_{op},\qquad C_L:=2/c_0.
\end{equation}

\smallskip
\emph{Step 3 (the family and the finite intersection).} Fix a basis
$v_1,\dots,v_m$ of $\mathfrak u$ with $\|v_i\|\le1$, $m=\dim U$, set
$u_i:=\exp v_i\in U$, and take $q_1:=e$, $q_{1+i}:=u_i$, so $J_G=m+1\le\dim G$; each
$q_iLq_i^{-1}$ is a Levi factor of $Q$. If $g\in L\cap\bigcap_i u_iLu_i^{-1}$,
then $u_i^{-1}gu_i\in L$ while $g^{-1}(u_i^{-1}gu_i)\in U$, so by $L\cap U=\{e\}$
the element $g$ commutes with $u_i$; hence $\operatorname{Ad}(g)v_i=v_i$ for all
$i$, so $\psi(g)=I$ and $g=e$ by Step~1. Thus the finite intersection is trivial,
and a fortiori so is the intersection of all Levi subgroups.

\smallskip
\emph{Step 4 (quantitative statement).} The local $UL$-coordinates
of $Q$ are smooth at $e$. Hence there are a neighbourhood of $e$
and $C_1\geq1$ such that if $q=w\alpha$ lies in this neighbourhood,
with $w\in U$ and $\alpha\in L$, then
\[
 d_G(e,w),\ d_G(e,\alpha)\leq C_1d_G(e,q).
\]
On a sufficiently small neighbourhood of each fixed $u_i$, the
logarithm on $U$ is Lipschitz. Also
$\|\Ad(\alpha^{-1})-I\|_{op}\leq C_1d_G(e,\alpha)$ for $\alpha$
near $e$, after increasing $C_1$. All these are local estimates;
we will apply them only to the error terms below.

Choose $l_0\in L$ and $l_i\in L$ with
$d_G(g,l_0)\leq\delta$ and
$d_G(g,u_il_iu_i^{-1})\leq\delta$. Such nearest points exist
because the metric is proper and the Levi subgroups are closed.
Right invariance gives
\[
 d_G(e,E_i)\leq2\delta,\qquad
 E_i:=(u_il_iu_i^{-1})l_0^{-1}.
\]
Its $UL$-decomposition is
\[
 E_i=w_i\alpha_i,\qquad
 w_i:=u_i(l_iu_i^{-1}l_i^{-1}),\qquad
 \alpha_i:=l_il_0^{-1}.
\]
For $\delta$ sufficiently small, both $w_i$ and $\alpha_i$ are
within $2C_1\delta$ of $e$. Since
$l_iu_il_i^{-1}=w_i^{-1}u_i$, right invariance gives
\[
 d_G(l_iu_il_i^{-1},u_i)=d_G(w_i^{-1},e)\leq2C_1\delta.
\]
Taking logarithms near $u_i=\exp v_i$, we obtain
$\|\Ad(l_i)v_i-v_i\|\leq C_3\delta$.
Now $l_0=\alpha_i^{-1}l_i$, so
\[
 \begin{split}
 \Ad(l_0)v_i-v_i
 &=\Ad(\alpha_i^{-1})(\Ad(l_i)v_i-v_i)\\
 &\qquad +(\Ad(\alpha_i^{-1})-I)v_i.
 \end{split}
\]
The bounds on $\alpha_i$ therefore give
$\|\Ad(l_0)v_i-v_i\|\leq C_5\delta$, with no bound assumed on
$l_0$ or $l_i$. Since the $v_i$ form a fixed basis of $\mathfrak u$,
\begin{equation}\label{eq:C6}
 \|\psi(l_0)-I\|_{op}\leq C_6\delta.
\end{equation}
Choose $\delta_0$ small enough for the local estimates above and
for $C_6\delta_0\leq\eta_0$. Equation~\eqref{eq:C_L} then gives
$d_G(e,l_0)\leq C_LC_6\delta$. Finally,
\[
 d_G(e,g)\leq d_G(g,l_0)+d_G(e,l_0)
             \leq(1+C_LC_6)\delta.
\]
Taking $D>1+C_LC_6$ puts $g$ in the open ball $G_{D\delta}$
when $\delta>0$. For $\delta=0$, the exact intersection statement
gives $g=e$.
\end{proof}

We proceed to a quantitative version of \cite[Lemma~6.4]{FG}.

\begin{lemma}\label{lem: factor final}
Let $G$ be centre-free and let $Q$ be a standard proper parabolic
containing $P$. There is $r_0=r_0(G)>0$ such that the following holds
for $0<r\leq r_0$, $0<\varepsilon<1$ and $\overline R\geq R_2(G)$.
Let $\lambda$ be norm-$\delta$-almost invariant under $P$, and suppose
\begin{gather}
 \lambda\{\Gamma:d_H^+(\Gamma^{\|}_{\overline R^b},Q)
                 \leq r/(4C_0D)\}\geq1-p,\label{eq: assumption again}\\
 \lambda\{\Gamma:\Ii(\Gamma)\geq r\}\geq1-\frac{\varepsilon}{16J_G},
 \qquad p<\frac{\varepsilon}{16J_G},\notag\\
 n_1\delta+p<
 \frac{\varepsilon/(32J_G)}
 {(16J_G\overline R^b r^{-\dim G}/\varepsilon+1)
       \log_2(4C_0D\overline R^b/r)}.\label{eq: factor numerical}
\end{gather}
Then
\[
 \lambda\{\Gamma:\Gamma\cap G_{\overline R}^{\|}=\{e\}\}>1-\varepsilon.
\]
\end{lemma}

\begin{proof}
Choose a Levi factor $L$ of $Q$ with split centre contained in $P$,
and the $J_G$ conjugates $q_iLq_i^{-1}$ of
Lemma~\ref{lem: trivial intersection}. Their split centres are also
contained in $P$, since $q_i$ belongs to the unipotent radical of $Q$.
Put
\[
 c:=\frac{\varepsilon}{2J_G},\qquad
 \delta':=\frac{r}{4C_0D}.
\]
The contraction argument in Lemma~\ref{lem: quantitative discreteness}
uses at most $n_1\delta$ for the defect under each conjugated contraction
element. Its hypotheses with radius $\overline R$ are exactly implied
by \eqref{eq: assumption again}--\eqref{eq: factor numerical}; in
particular $p<c/8$ and the exceptional mass for $r$-discreteness is at
most $c/8$. Shrink $r_0$ so that $\delta'$ is within the range of
Lemma~\ref{lem: quantitative discreteness} and $2C_0\delta'$ is within
the range of Lemma~\ref{lem: trivial intersection}.
Applying the argument to each of the $J_G$ Levi factors gives
\[
 \lambda\{\Gamma:d_H^+(\Gamma_{\overline R}^{\|},q_iLq_i^{-1})
                   \leq2C_0\delta'\text{ for every }i\}
 \geq1-J_Gc=1-\varepsilon/2.
\]
Lemma~\ref{lem: trivial intersection} puts every element of
$\Gamma_{\overline R}^{\|}$ in $G_{2DC_0\delta'}=G_{r/2}$.
On the $r$-discrete event this intersection is $\{e\}$.
The excluded mass is at most
$\varepsilon/2+\varepsilon/(16J_G)<\varepsilon$.
\end{proof}

\subsection{Almost Furstenberg decomposition and conclusion of the proof}

\begin{lemma}[Uniform discreteness]\label{lem: margulis functions}
With $A$ as in (C4), every discrete $\Lambda\leq G$ satisfies
\[
 (A_\mu^n\delta_\Lambda)\{\Gamma:\Ii(\Gamma)\leq u\}
 \leq A(\Ii(\Lambda))u^\varkappa\qquad(n\geq1,\ u>0).
\]
\end{lemma}
\begin{proof}
Iteration of \eqref{eq:margulis-drift} gives
\[
 \int_G\Ii(\Lambda^g)^{-\varkappa}\,d\mu^{*i}(g)
 \leq c_0^i\Ii(\Lambda)^{-\varkappa}+\frac{b_0}{1-c_0}
 \leq A(\Ii(\Lambda)).
\]
Apply Markov's inequality and average over $1\leq i\leq n$.
\end{proof}

For the norm estimates put $\Theta(g):=|g|_{\Ad}^2$.
This bounds conjugation and its inverse both for $d_G$ and for the
size of norm balls.
\begin{claim}\label{cl: aux}
For subsets $S_1,S_2\subseteq G$, conjugation satisfies
\[
 d_H^+(gS_1g^{-1},gS_2g^{-1})
 \leq\|\Ad(g)\|d_H^+(S_1,S_2).
\]
The same holds for the two-sided Hausdorff distance. Moreover, for $R>1$,
\[
 gG_R^{\|}g^{-1}\subseteq G_{\Theta(g)R}^{\|}.
\]
Thus $\Theta(g)$ bounds this distortion and the enlargement of a
norm ball under conjugation.
\end{claim}
\begin{proof}
Conjugation is left translation by $g$ followed by right translation by
$g^{-1}$. Right translation is an isometry for $d_G$, and the norm of
the differential of left translation is $\|\Ad(g)\|$.
The matrix-norm assertion follows from submultiplicativity.
\end{proof}

We use the almost Furstenberg decomposition from
\cite[Thm. ]{QNZ}. For $t>0$ it gives
a probability $\sigma_t$ on $P$ with
\begin{equation}\label{eq: one-sided furstenberg}
 p_t\,dm_G=m_K*\sigma_t,
 \qquad
 \sup_{p\in C}\|p\sigma_t-\sigma_t\|_{\mathrm{TV}}
 \leq A_Ct^{-1/4}\quad(t\geq1)
\end{equation}
for every compact $C\subset P$.

\begin{lemma}[From a point to a measure]\label{lem: upgraded furstenberg}
There are constants $C_*,A_G,C_G>0$, depending only on $G$, such that,
for $m\geq1$,
\[
 \overline\mu_m:=\frac1m\sum_{i=1}^m\mu^{*i}
 =\int_K\theta\widetilde\lambda_\theta\,dm_K(\theta),
\]
where the probabilities $\widetilde\lambda_\theta$ are supported on $P$,
are norm-$A_Gm^{-1/4}$-almost invariant under $P$, and satisfy
\[
 \widetilde\lambda_\theta\{p:\Theta(p)>e^{C_*m}\}\leq C_Ge^{-m/C_G}\leq C_G'/m.
\]
If $\Lambda$ is discrete and $\lambda_\theta:=\widetilde\lambda_\theta*\delta_\Lambda$,
then
\[
 A_\mu^m\delta_\Lambda=\int_K\theta\lambda_\theta\,dm_K(\theta).
\]
Moreover, with $a_r:=(A(\Ii(\Lambda))r^\varkappa)^{1/2}$, there is
$K_0\subseteq K$ of Haar measure at least $1-a_r$ such that
\[
 \lambda_\theta\{\Gamma:\Ii(\Gamma)\geq r\}\geq1-a_r
 \qquad(\theta\in K_0).
\]
\end{lemma}

\begin{proof}
Set $\overline\sigma_m:=m^{-1}\sum_{i=1}^m\sigma_{N_0i}$ and take
$\widetilde\lambda_\theta:=\overline\sigma_m$ for every $\theta$.
Equation~\eqref{eq: one-sided furstenberg} gives the exact decomposition.
For the compact set $\overline{P\cap G_1}$, its invariance estimate gives
\[
 \sup_{p\in P\cap G_1}\|p\overline\sigma_m-\overline\sigma_m\|_{\mathrm{TV}}
 \leq\frac{A_C}{m}\sum_{i=1}^m(N_0i)^{-1/4}
 \leq A_Gm^{-1/4}.
\]

Our norm convention gives $\Theta(\theta p)=\Theta(p)$ for
every $\theta\in K$.
The matrix norm also satisfies $\Theta(g)\leq B e^{B d_G(e,g)}$.
Choose $C_*$ sufficiently large. If $\Theta(p)>e^{C_*m}$, every
$\theta p$ lies outside $G_{v_Gm}$. The decomposition and
Lemma~\ref{lem:concentration} therefore give the asserted exponential
tail, after including the fixed constants. Convolution with
$\delta_\Lambda$ preserves almost invariance and gives the displayed
decomposition on $\subd$.

Finally, put $W_r:=\{\Gamma:\Ii(\Gamma)\geq r\}$.
Using the $K$-invariance of $W_r$ and Lemma~\ref{lem: margulis functions},
\[
 \int_K\lambda_\theta(W_r^c)\,dm_K(\theta)
 =(A_\mu^m\delta_\Lambda)(W_r^c)
 \leq A(\Ii(\Lambda))r^\varkappa.
\]
Here $\lambda_\theta$ does not depend on $\theta$. If $a_r\leq1$,
the preceding bound is at most $a_r$; if $a_r>1$, the assertion is
trivial. Thus one may take $K_0=K$.
\end{proof}

\begin{corollary}\label{cor: fin cor}
Suppose $Q\geq P$ is proper parabolic, $R>T_m$, and
$d_H^+(\Lambda\cap G_R^{\|},Q)\leq\delta$.
For $\theta\in K_0$ as in Lemma~\ref{lem: upgraded furstenberg},
\begin{gather*}
 \lambda_\theta\{\Gamma:d_H^+(\Gamma\cap G_{R/T_m}^{\|},Q)
                           \leq T_m\delta\}\geq1-p_m,\\
 \lambda_\theta\{\Gamma:\Ii(\Gamma)\geq r\}\geq1-a_r.
\end{gather*}
\end{corollary}
\begin{proof}
For $p\in P$ with $\Theta(p)\leq T_m$, conjugation sends
$\Lambda\cap G_R^{\|}$ to a set containing
$p\Lambda p^{-1}\cap G_{R/T_m}^{\|}$, and increases the proximity
error by at most $T_m$. Since $pQp^{-1}=Q$, this gives the first
assertion outside the exceptional set of mass $p_m$.
The second is part of the preceding lemma.
\end{proof}

\begin{proof}[Proof of Lemma~\ref{lem: almost factor}]
Write $\Lambda:=\Stab_G(x)$. Conjugating by an element of $K$ and
enlarging the proper parabolic if necessary, we may suppose that
$Q$ is a standard maximal parabolic containing $P$.
Put $z:=\log R$ and
\[
 m:=\left\lfloor\frac{z-\tfrac18\log z}{C_*}\right\rfloor,
 \qquad \overline R:=(R/T_m)^{1/b}.
\]
The floor gives
\begin{equation}\label{eq: factor radius choice}
 z^{1/8}\leq\overline R^b\leq e^{C_*}z^{1/8},
 \qquad m\asymp_G z.
\end{equation}
We fix $r_0$ small enough that the size assumption in the lemma ensures
$m\geq1$ and $\overline R\geq R_2(G)$.

Apply Corollary~\ref{cor: fin cor}. On $K_0$ the measures
$\lambda_\theta$ are norm-$\overline\delta$-almost invariant under $P$,
where $\overline\delta:=A_Gm^{-1/4}$, and the parabolic error is
$T_m\delta\leq R\delta<c_*r=r/(4C_0D)$.
Take $\varepsilon:=\eta'/2$ in Lemma~\ref{lem: factor final}.
The admissibility assumption gives
$a_r\leq\eta'/(32J_G)=\varepsilon/(16J_G)$.

It remains to check \eqref{eq: factor numerical} and
$p_m<\varepsilon/(16J_G)$ uniformly in $\eta'$ and $r$.
By \eqref{eq: factor radius choice}, it suffices to have
\begin{equation}\label{eq: factor size verification}
 C\eta'^{-2}r^{-\dim G}
       \bigl(\log z+\log(1/r)\bigr)<z^{1/8},
\end{equation}
with a fixed constant $C=C(G)$. Indeed, the invariance defect is
$O_G(z^{-1/4})$ and the denominator in \eqref{eq: factor numerical}
is $O_G(\eta'^{-1}z^{1/8}r^{-\dim G})$ times the displayed logarithmic factor; the tail $p_m$ is
exponentially small in $z$.

To verify \eqref{eq: factor size verification}, set
$h:=\log(1/r)\geq1$ and $a:=\eta'^{-2}r^{-\dim G}$.
The hypothesis is $z\geq(ah)^9$. Once $r_0$ is small enough,
$(\log z+h)z^{-1/8}$ decreases throughout this range, and
\[
 \frac{a(\log z+h)}{z^{1/8}}
 \leq a^{-1/8}h^{-9/8}(9\log a+9\log h+h).
\]
This tends uniformly to zero as $r_0\to0$, because
$a\geq r^{-\dim G}$; the same bound covers all $0<\eta'<1$.
Fix $r_0$ so that the right side is below $C^{-1}$, and so that the
exponential tail satisfies the preceding bound on $p_m$.
This proves all the hypotheses of Lemma~\ref{lem: factor final}.

For every $\theta\in K_0$ its conclusion gives
\[
 \lambda_\theta\{\Gamma:\Gamma\cap G_{\overline R}^{\|}=\{e\}\}
 >1-\eta'/2.
\]
Since $\overline R\geq z^{1/(8b)}=z^\beta$, integration over $K_0$
and $K$-invariance of the norm-ball event give
\[
 (A_\mu^m\delta_\Lambda)
 \{\Gamma:\Gamma\cap G_{(\log R)^\beta}^{\|}=\{e\}\}
 \geq(1-\eta'/2)(1-\eta'/(32J_G))\geq1-\eta'.
\]
\end{proof}

\section{Deriving the consequences}
\label{sec: consequences}

This last section is concerned with recording the proofs of the consequences of the main theorem (Theorem \ref{thm:genesis}).

\subsection{A quantitative Fraczyk--Gelander theorem}\label{sec: qfg}

We first record the consequence of Theorem~\ref{thm:genesis} for discrete subgroups of infinite covolume. 

\begin{theorem}\label{thm: qfg consequence}
Let $G$ be a connected simple Lie group with finite centre and
$\mathrm{rank}_{\RR}G\geq2$, put $X=K\backslash G$, and let
$\Gamma\leq G$ be a discrete subgroup of infinite covolume. Fix
$o\in M:=X/\Gamma$. There are $c>0$ and $r_0>0$, depending only on
$G,\Gamma,o$, such that, for every $r\geq r_0$, some $p\in B_r^M(o)$
satisfies
\[
 \operatorname{inj}_M(p)\geq c\log^{(4)}r.
\]
Moreover, if $\nu_n^o$ is the Ces\`aro heat-kernel average from $o$
at the times used in Theorem~\ref{thm:genesis}, then
\[
 \nu_n^o\{p:\operatorname{inj}_M(p)\geq c\log^{(4)}n\}\geq p_G
 \qquad(n\geq r_0),
\]
where $p_G>0$ depends only on $G$.
\end{theorem}

\begin{proof}
Suppose first that $G$ has trivial centre. Choose a lift $g_0\in G$
of $o$ and apply Theorem~\ref{thm:genesis} with $k=1$ to
$\Lambda=g_0\Gamma g_0^{-1}$. Since $\Lambda$ is not a lattice, alternative 2 of Theorem \ref{thm:genesis} implies that for all sufficiently large $n$, a proportion at least
$p_G$ of the heat-kernel average on $G/\Lambda$ satisfies
\[
 \Ii(\operatorname{Stab}_G(x))\geq c_G\log^{(4)}(n/2).
\]
Under the equivariant identification $h\Lambda\mapsto hg_0\Gamma$,
equation~\eqref{eq: symmetric injectivity comparison} implies
\begin{equation}
\operatorname{inj}_M([hg_0])
 \geq \tfrac12\bigl(c_G\log^{(4)}(n/2)-D_K\bigr)_+. \label{Eq: thick}
\end{equation}

Decreasing the radius constant and increasing the threshold for $n$
gives the asserted probability bound for $\nu_n^{o}$ in the second assertion.

By Lemma~\ref{lem:concentration}, the same heat-kernel average gives mass at least
$1-C_He^{-c_Hn}$ to the ball $G_{v_Gn}$. Its intersection with the
set of points satisfying \eqref{Eq: thick} therefore has positive mass once $C_He^{-c_Hn}<p_G$.
For an element $h$ in this intersection, right invariance and the
quotient metric give
\[
 d_M(o,[hg_0])\leq d_G(g_0,hg_0)=d_G(e,h)<v_Gn.
\]
Taking $n=\lfloor r/(2\max\{1,v_G\})\rfloor$ and using
$\log^{(4)}n\geq\tfrac12\log^{(4)}r$ for sufficiently large $r$ finally
proves the first assertion of Theorem \ref{thm: qfg consequence}. Finally, for simple Lie groups with finite centre, it suffices to pass to the adjoint
quotient as in Lemma~\ref{lem: geometric comparisons}.
\end{proof}

\subsection{Speed of Benjamini--Schramm convergence}
\label{sec: quantitative bs}

We next give a quantitative version of Benjamini--Schramm convergence
for lattices of growing volume, as in \cite[Corollary~4.7]{7s17}.
The passage from heat-kernel averages to normalised volume uses
invariance and the following straightforward consequence of the Margulis inequality which we record for completeness.

\begin{lemma}[Discreteness under an invariant probability]
\label{lem: invariant discreteness consequences}
Let $\xi$ be a conjugation-invariant probability on $\subd$, and put
$B_G=b_0/(1-c_0)$, with the constants of
\eqref{eq:margulis-drift}. Then, for every $u>0$,
\begin{equation}\label{eq: invariant small radius}
 \xi\left(\{\Lambda:\Ii(\Lambda)<u\}\right)\leq B_Gu^{\varkappa}.
\end{equation}
\end{lemma}

\begin{proof}
Set $F(\Lambda)=\Ii(\Lambda)^{-\varkappa}$ and let $\mathcal P$
be the heat-kernel Markov operator on subgroups. Iterating
\eqref{eq:margulis-drift} gives
\[
 \mathcal P^jF(\Lambda)\leq c_0^jF(\Lambda)+B_G.
\]
Thus, for $E_u=\{\Lambda:\Ii(\Lambda)<u\}$, Markov's inequality gives
\[
 \mathcal P^j\mathbf1_{E_u}(\Lambda)
 \leq\min\{1,u^{\varkappa}(c_0^jF(\Lambda)+B_G)\}.
\]
Integrate against $\xi$ and use its invariance. The left side has
integral $\xi(E_u)$ for every $j$. On the right, dominated convergence
applies because the integrand is bounded by $1$ and $F$ is finite on
discrete subgroups. Letting $j\to\infty$ proves
\eqref{eq: invariant small radius}. This argument requires no
integrability assumption on $F$.
\end{proof}

We now prove the Benjamini--Schramm convergence with effective rate.

\begin{theorem}[Uniform quantitative Benjamini--Schramm convergence]
\label{thm: quantitative bs}
Let $G$ be a connected centre-free simple Lie group with
$\mathrm{rank}_{\RR}G\geq2$, let $X=K\backslash G$, and let
$\Gamma\leq G$ be a lattice. Put $M=X/\Gamma$ and
$V=\operatorname{vol}(M)$, with the fixed metric normalisation.
There is $C_G\geq1$, depending only on $G$, such that, for every $R>0$,
\begin{equation}\label{eq: bs logstar bound}
 \frac{\operatorname{vol}(M_{\leq R})}{V}
 \leq
 \min\left\{1,
 \frac{C_G}{\bigl(1+(\log^*V-\log^*(R+1))_+\bigr)^{1/8}}
 \right\},
\end{equation}
where
$M_{\leq R}=\{p\in M:\operatorname{inj}_M(p)\leq R\}$ and
$t_+=\max\{t,0\}$. In particular, for every fixed $R$,
\[
 \frac{\operatorname{vol}(M_{\leq R})}{V}
 =O_G\bigl((\log^*V)^{-1/8}\bigr)
 \qquad(V\to\infty),
\]
valid for $V$ large enough depending on $G$ and $R$ alone.
\end{theorem}

\begin{proof}
Write $Y=G/\Gamma$, let $\sigma$ be normalised Haar measure on $Y$,
and let $\pi:Y\to M$ be the quotient map. Then $\pi_*\sigma$ is the
normalised Riemannian volume. For $x\in Y$, put
\[
 \Lambda_x=\operatorname{Stab}_G(x),\qquad
 \nu_n^x=\frac1n\sum_{i=1}^n\mu^{*i}*\delta_x.
\]
Invariance gives, for every measurable $E\subseteq Y$,
\begin{equation}\label{eq: bs invariant averaging}
 \int_Y\nu_n^x(E)\,d\sigma(x)=\sigma(E).
\end{equation}
The stabiliser pushforward of $\sigma$ satisfies
Lemma~\ref{lem: invariant discreteness consequences}; hence
\begin{equation}\label{eq: bs initial exception}
 \sigma\left(\{x:\Ii(\Lambda_x)<u\}\right)\leq B_Gu^{\varkappa}.
\end{equation}

Put $d=\dim G$, $a_0=9(d+1)$ and
\[
 \alpha=\frac{\varkappa}{9(d+1)}.
\]
Let $P$ be the exponent in Theorem~\ref{thm:genesis}. Its bound on
$D_\Lambda$ implies, for $0<u\leq1$,
\begin{equation}\label{eq: bs uniform genesis threshold}
 \Ii(\Lambda_x)\geq u
 \quad\Longrightarrow\quad
 D_{\Lambda_x}
 \leq C_G(1+u^{-\varkappa})^{a_0/\varkappa}
 \leq D_0u^{-a_0},
\end{equation}
where $D_0$ depends only on $G$.

Let $A_G\geq2$ be a constant to be fixed depending on $G$ alone.
For integers $n,k\geq1$, set $L=\log^{(3k)}(n/2)$. If these
iterates are defined and
\begin{equation}\label{eq: bs finite hypotheses}
 V>A_Ge^{A_Gn},\qquad
 L\geq\max\{A_Gk^P,e^{A_G(R+1)}\},
\end{equation}
then
\begin{equation}\label{eq: bs finite bound}
 \frac{\operatorname{vol}(M_{\leq R})}{V}
 \leq C_G\left[k^{-1/8}
       +\left(\frac{k^P}{L}\right)^{\alpha}\right].
\end{equation}

Indeed, choose
\[
 u=\left(\frac{A_Gk^P}{L}\right)^{1/a_0}\leq1.
\]
For every $x$ with $\Ii(\Lambda_x)\geq u$, equation
\eqref{eq: bs uniform genesis threshold} gives
$D_{\Lambda_x}k^P\leq L$ if $A_G$ is large enough. Apply Theorem~\ref{thm:genesis} to $G/\Lambda_x = Y$. Alternative 1 of Theorem~\ref{thm:genesis} is excluded by the first
inequality in \eqref{eq: bs finite hypotheses}, after absorbing the
fixed volume factor in \eqref{eq: volume normalization} into $A_G$. Thus
\[
 \nu_n^x\left(\{y:\Ii(\Lambda_y)\geq c_G\log L\}\right)
 \geq1-C_Gk^{-1/8}
\]
by alternative 2 of Theorem~\ref{thm:genesis}. By \eqref{eq: symmetric injectivity comparison}, these points satisfy
\[
 \operatorname{inj}_M(\pi(y))
 \geq\tfrac12(c_G\log L-D_K)_+>R,
\]
provided $A_G$ was chosen so that $c_GA_G\geq2+D_K$. Apply \eqref{eq: bs invariant averaging} with
$E=\pi^{-1}(M_{\leq R})$ and use \eqref{eq: bs initial exception} to obtain
\[
 \begin{aligned}
 \frac{\operatorname{vol}(M_{\leq R})}{V}&\leq \int_{\{x : \mathcal{I}(\Lambda_x) \geq u\}} \nu_n^x(E) d\sigma(x)+ \sigma(\{x : \mathcal{I}(\Lambda_x) < u\})\\
 &\leq C_Gk^{-1/8}+B_Gu^{\varkappa}\\
 &\leq C_G\left[k^{-1/8}
       +\left(\frac{k^P}{L}\right)^\alpha\right],
 \end{aligned}
\]
after increasing $C_G$. This proves \eqref{eq: bs finite bound}.

We now choose the parameters in terms of $V$ and $R$. For $V$
sufficiently large depending only on $G$, set
\[
 n_V=\left\lfloor\frac{\log(V/A_G)}{2A_G}\right\rfloor,
 \qquad t=n_V/2,\qquad N=\log^*t,
 \qquad H=\log^*\bigl(e^{A_G(R+1)}\bigr).
\]
Then
$A_Ge^{A_Gn_V}\leq\sqrt{A_GV}<V$ when $V$ is large enough. Put $k=\lfloor(N-H)/4\rfloor$ whenever this is positive. If
$k\geq1$, the definition of $\log^*$ gives
\begin{equation}\label{eq: bs remaining iterates}
 L=\log^{(3k)}t
 >\exp^{(N-3k-1)}(e)
 \geq\exp^{(H+k-1)}(e).
\end{equation}
In particular, $L\geq e^{A_G(R+1)}$ and
$L\geq\exp^{(k-1)}(e)$. There is an integer $k_0=k_0(G)$ such that,
for every $k\geq k_0$,
\[
 \exp^{(k-1)}(e)\geq A_Gk^P,
 \qquad
 \left(\frac{k^P}{\exp^{(k-1)}(e)}\right)^\alpha\leq k^{-1/8}.
\]
Consequently \eqref{eq: bs finite bound} is at most $2C_Gk^{-1/8}$
whenever $k\geq k_0$.

For each fixed $a\geq1$, notice that the difference $\log^*(as)-\log^*s$
is bounded above by a constant depending only on $a$. 
For $V$ sufficiently large one also has
$\log V/(16A_G)\leq t\leq\log V$. It follows that
\[
 N\geq\log^*V-O_G(1),\qquad
 H\leq\log^*(R+1)+O_G(1).
\]
Thus, if $\log^*V-\log^*(R+1)$ is greater than a constant depending on $G$ alone, then $k=\lfloor(N-H)/4\rfloor$ is at least $k_0$. The preceding estimate proves
\eqref{eq: bs logstar bound} for such a choice of $V$ and $R$. One may adjust $C_G$ so that the asserted bound \eqref{eq: bs logstar bound} is
the trivial bound $1$ for all other values of $V$ and $R$. The "In particular" assertion follows directly.
\end{proof}

For completeness, we also record a form with prescribed error: the same computation gives an explicit
threshold without suppressing the tower parameters.

\begin{corollary}[A volume threshold for a prescribed error]
\label{cor: bs tower threshold}
Use the constants $A_G,C_G,P,\alpha$ in
\eqref{eq: bs finite hypotheses}--\eqref{eq: bs finite bound}.
For $R>0$ and $0<\varepsilon<1$, define
\[
 \begin{aligned}
 k_\varepsilon
   &=\left\lceil(2C_G/\varepsilon)^8\right\rceil,\\
 T_{R,\varepsilon}
   &=\max\left\{
       A_Gk_\varepsilon^P(2C_G/\varepsilon)^{1/\alpha},
       e^{A_G(R+1)}\right\},\\
 n_{R,\varepsilon}
   &=\left\lceil2\exp^{(3k_\varepsilon)}(T_{R,\varepsilon})
       \right\rceil,\\
 V_0(R,\varepsilon)
   &=A_G\exp\bigl(A_G(n_{R,\varepsilon}+1)\bigr).
 \end{aligned}
\]
If $\operatorname{vol}(M)\geq V_0(R,\varepsilon)$, then
\[
 \frac{\operatorname{vol}(M_{\leq R})}{\operatorname{vol}(M)}
 \leq\varepsilon.
\]
\end{corollary}

\begin{proof}
Take $n=n_{R,\varepsilon}$ and $k=k_\varepsilon$. Then
$L=\log^{(3k)}(n/2)\geq T_{R,\varepsilon}$, and the volume
assumption implies $V>A_Ge^{A_Gn}$. Each term on the right of
\eqref{eq: bs finite bound} is at most $\varepsilon/2$.
\end{proof}

\subsection{The Stuck--Zimmer theorem}\label{sec: stuck zimmer}

The proof of Theorem~\ref{thm:genesis} uses the structure theorem of
\cite{QNZ} and, contrary to \cite{FG}, does not invoke the Stuck--Zimmer theorem. We can
therefore deduce the latter from Theorem~\ref{thm:genesis}. The direct approach through
the companion paper is discussed in Section~\ref{Subsec: A short proof}.

First recall the standard reduction to the classification of invariant probably measures on $\subd$.

\begin{proposition}[Reduction to discrete invariant random subgroups]
\label{prop: subdg}
Let $G$ be a connected, noncompact, centre-free simple Lie group.
Suppose that every ergodic invariant probability on $\subd$ is either
$\delta_{\{e\}}$ or
\[
 \nu_\Gamma=(\iota_\Gamma)_*\overline m_{G/\Gamma},\qquad
 \iota_\Gamma(g\Gamma)=g\Gamma g^{-1},
\]
for some lattice $\Gamma$, where $\overline m_{G/\Gamma}$ is normalised
Haar measure. Then every ergodic probability-preserving action of $G$
on a standard probability space is essentially free or essentially transitive.
\end{proposition}

%

And a simple application of the ergodic theorem.

\begin{claim}\label{cl: stat ergodicity}
Let $\nu$ be an ergodic invariant probability on $\subd$. For
$\nu$-almost every $\Lambda$, the measures
\[
 \rho_{n,\Lambda}=\frac1n\sum_{i=1}^n\mu^{*i}*\delta_\Lambda
\]
converge weakly to $\nu$ in the space of probabilities on $\operatorname{Sub}(G)$.
\end{claim}

\begin{proof}
Apply \cite[Theorem~2, equation~(3.23)]{kakutani1951random} to the
heat-kernel Markov operator $\mathcal P$ and a countable uniformly
dense family in $C(\operatorname{Sub}(G))$. These functions are bounded.
If $\mathcal Pf=f$ in $L^2(\nu)$, invariance of $\nu$ gives
\[
 \int_G\|f\circ g-f\|_2^2\,d\mu(g)
 =2\|f\|_2^2-2\operatorname{Re}\langle\mathcal Pf,f\rangle=0.
\]
Full support of $\mu$ and strong continuity on $L^2(\subd, \nu)$ then make
$f$ $G$-invariant, hence constant by ergodicity. The pointwise limits
are therefore the integrals against $\nu$. Intersecting the conull
sets for the countable family proves the asserted weak convergence.
\end{proof}

We can now deduce the proof of Theorem \ref{thm: C}.

\begin{proof}[Proof of Theorem~\ref{thm: C}]
By Proposition~\ref{prop: subdg}, it suffices to classify an ergodic
invariant probability $\nu$ on $\subd$. Choose $\Lambda$ as in
Claim~\ref{cl: stat ergodicity}.

If $\Lambda$ is a lattice, the Ces\`aro heat-kernel averages on
$G/\Lambda$ converge in total variation to normalised Haar measure.
Indeed, after the first step the law has an $L^1$ density. The 
ergodic theorem gives convergence in $L^1$ on $G/\Gamma$ of the Ces\`aro averages to a stationary measure. But the Haar measure is the only stationary measure. This concludes the first case.

Otherwise apply Theorem~\ref{thm:genesis} with $k=1$. For all
sufficiently large $n$, alternative 1. is excluded. Thus, for some $T_n\to\infty$,
\[
 \rho_{n,\Lambda}\left(\{H:H\cap G_{T_n}=\{e\}\}\right)\geq p_G.
\]
For $R>0$, set
\[
 F_R=\{H\in\operatorname{Sub}(G):H\cap G_R=\{e\}\}.
\]
This set is closed in the Chabauty topology: its complement is the
set of closed subgroups meeting the open set $G_R\setminus\{e\}$.
For fixed $R$, eventually $T_n\geq R$. Portmanteau therefore gives
\[
 \nu(F_R)\geq\limsup_{n\to\infty}\rho_{n,\Lambda}(F_R)\geq p_G.
\]
Since $F_j\downarrow\{\{e\}\}$ as the integer $j\to\infty$,
continuity from above gives $\nu(\{\{e\}\})\geq p_G$. This singleton
is invariant, so ergodicity implies $\nu=\delta_{\{e\}}$.
\end{proof}
\appendix

\section{Geometric comparisons}\label{app: geometric comparisons}

We use the metrics and matrix size fixed in Section~\ref{sec: def}.
The following comparisons also specify the effect of the compact
factor when passing from discreteness to injectivity radius.

\begin{lemma}\label{lem: geometric comparisons}
Put $D_K:=\max_{k\in K}d_G(e,k)$ and
$r_X(g):=d_X(K,Kg)$. Then
\begin{equation}\label{eq: group Cartan comparison}
 r_X(g)=\|\kappa(g)\|\leq d_G(e,g)
              \leq r_X(g)+D_K.
\end{equation}
There is $C_{\rm met}\geq1$, depending only on $G$, such that
\begin{equation}\label{eq: norm metric comparison}
 G_s\subseteq G_{e^{C_{\rm met}s}}^{\|}\quad(s>0),
 \qquad
 G_R^{\|}\subseteq G_{C_{\rm met}\log R}\quad(R\geq2).
\end{equation}
If $G$ has trivial centre and $\Gamma\leq G$ is discrete, then
\begin{equation}\label{eq: symmetric injectivity comparison}
 \begin{split}
 \on{inj}_{X/\Gamma}([g])
 &=\frac12\inf_{\gamma\in\Gamma\setminus\{e\}}
          r_X(g\gamma g^{-1})\\
 &\geq\frac12\bigl(\Ii(g\Gamma g^{-1})-D_K\bigr)_+.
 \end{split}
\end{equation}
The infimum over the empty set is $+\infty$. For finite-centre $G$,
the same injectivity statement holds with $G,K,\Gamma,g$ replaced
by their images in the adjoint quotient.
\end{lemma}

\begin{proof}
The quotient metric gives
$r_X(g)=\inf_{k\in K}d_G(e,kg)$. Right invariance and the triangle
inequality imply
\[
 d_G(e,g)\leq d_G(e,kg)+d_G(kg,g)
            \leq d_G(e,kg)+D_K.
\]
Taking the infimum proves the inequalities in
\eqref{eq: group Cartan comparison}; the equality with
$\|\kappa(g)\|$ is the Cartan distance formula for the chosen metric.

The differential of the adjoint representation gives
$|g|_{\Ad}\leq e^{C d_G(e,g)}$, by integration along a path from
$e$ to $g$. Conversely, for $g=k_1\exp(H)k_2$ with $H$ in the
closed positive chamber, $K$-invariance of the norm gives
\[
 \log|g|_{\Ad}=\max_{\alpha\in\Sigma}|\alpha(H)|.
\]
The restricted roots span $\mathfrak a^*$, so this is bounded below
by a positive constant times $\|H\|$. Equation
\eqref{eq: group Cartan comparison} absorbs the remaining $D_K$
when $R\geq2$, proving \eqref{eq: norm metric comparison}.

Write $x:=Kg$. Right translation is isometric on $X$, and hence
\[
 d_X(x,x\gamma)=d_X(K,Kg\gamma g^{-1}).
\]
If all these displacements are at least $2r$, the open ball
$B_r^X(x)$ injects into $X/\Gamma$: two identified points in the
ball would give a displacement of $x$ strictly smaller than $2r$.
Conversely, if $d_X(x,x\gamma)<2r$ for some nontrivial $\gamma$,
the balls $B_r^X(x)$ and $B_r^X(x\gamma)$ have an open intersection.
The action is effective, so the fixed-point set of $\gamma$ has
empty interior. Choose a point of this intersection not fixed by
$\gamma$; it and its translate by $\gamma^{-1}$ are distinct points
of $B_r^X(x)$ with the same image. This proves the equality in
\eqref{eq: symmetric injectivity comparison}, also for torsion
elements. Its inequality follows from
\eqref{eq: group Cartan comparison}.

Finally, the finite central quotient is proper, so it sends a
discrete subgroup to a discrete subgroup. Its inverse image is
$\Gamma Z(G)$, a finite extension of $\Gamma$; the usual quotient
integration formula shows that finite covolume is equivalent for
$\Gamma$, $\Gamma Z(G)$ and $\Gamma_{\mathrm{ad}}$. Since $Z(G)\subset K$,
the symmetric space and its quotient metric are unchanged. This
proves the final assertion.
\end{proof}

\section{Heat-kernel estimates and regularization}\label{app: heat-kernel computations}
\subsubsection{Global heat-kernel estimates}
\label{app:heat-kernel-estimates}

We retain the metric and heat-time conventions of Section~\ref{sec: def}.
Write
\[
 q_n:=p_{N_0n},\qquad r_X(g):=d_X(K,Kg)=\|\kappa(g)\|,
 \qquad
 \varrho:=\frac12\sum_{\alpha\in\Sigma^+}m_\alpha\alpha.
\]
We identify $\varrho\in\mathfrak a^*$ with a vector in $\mathfrak a$.
In this appendix, $q_n$ denotes the density of a single heat-kernel step
of duration $N_0n$; we write $\overline q_n:=n^{-1}\sum_{i=1}^nq_i$
for the Ces\`aro density.

We use the global estimate of Anker--Ostellari~\cite{AnkerOstellari},
in the form recorded in
\cite[\S2.3, equations (2.14)--(2.16)]{APZ}:
\begin{equation}\label{eq:sharp-heat-kernel}
 \begin{split}
 p_t(\exp H)&\asymp_G F_t(H),\\
 F_t(H)&:=t^{-\dim X/2}P_t(H)
   e^{-\|\varrho\|^2t-\varrho(H)-\|H\|^2/(4t)},\\
 P_t(H)&:=\prod_{\alpha\in\Sigma_0^+}
   (1+\alpha(H))(1+t+\alpha(H))^{(m_\alpha+m_{2\alpha})/2-1}.
 \end{split}
\end{equation}
Here $t>0$, $H\in\overline{\mathfrak a^+}$, and
$\Sigma_0^+$ is the set of positive indivisible restricted roots.
The comparison constants are uniform in $t$ and $H$.

\subsubsection{Heat-kernel concentration}
\label{app:heat-kernel-drift-concentration}

We use $q_n,r_X,\varrho$ as in
\S\ref{app:heat-kernel-estimates}. The only additional
input is the Cartan integration formula, whose Jacobian is
$J(H)=\prod_{\alpha\in\Sigma^+}(\sinh\alpha(H))^{m_\alpha}$; see
\cite[\S2.1]{APZ}.

\begin{proposition}[Exponential concentration around the Cartan drift]
\label{prop:heat-kernel-cartan-concentration}
For every $\varepsilon>0$ there are $C_\varepsilon,a_\varepsilon>0$,
depending only on $G$ and $\varepsilon$, such that
\[
 \mu^{*n}\{g:\|\kappa(g)-2N_0n\varrho\|\geq\varepsilon n\}
 \leq C_\varepsilon e^{-a_\varepsilon n}\qquad(n\geq1).
\]
The same bound holds for
$\mu^{*n}\{g:|r_X(g)-2N_0\|\varrho\|n|\geq\varepsilon n\}$.
Consequently, if $0<c<2N_0\|\varrho\|<c'$, then
\[
 \mu^{*n}\{g:cn<r_X(g)<c'n\}\geq1-Ce^{-an}
\]
for constants $C,a>0$ depending on $G,c,c'$.
\end{proposition}
\begin{proof}
Since $J(H)\leq C e^{2\varrho(H)}$, the polynomial factors in
\eqref{eq:sharp-heat-kernel} can be absorbed into
$C n^M(1+\|H\|)^M$. Completing the square gives
\begin{equation}\label{eq:radial-gaussian}
 q_n(\exp H)J(H)
 \leq C n^M(1+\|H\|)^M
       e^{-\|H-2N_0n\varrho\|^2/(4N_0n)}.
\end{equation}
Integrate over $\|H-2N_0n\varrho\|\geq\varepsilon n$ and set
$Z:=(H-2N_0n\varrho)/\sqrt{N_0n}$. The resulting Gaussian tail is
exponentially small in $n$; its polynomial prefactor is absorbed by
decreasing $a_\varepsilon$. The radial conclusion follows from
\[
 |\|\kappa(g)\|-2N_0\|\varrho\|n|
 \leq\|\kappa(g)-2N_0n\varrho\|.
\]
Choose $\varepsilon<\min\{2N_0\|\varrho\|-c,c'-2N_0\|\varrho\|\}$
for the final assertion.
\end{proof}

\begin{lemma}[Uniform concentration for the averages]\label{lem:concentration}
There are $v_G,C_H,c_H>0$, depending only on $G$, such that
\begin{equation}\label{eq:concentration}
 \mu^{*i}(G_{v_Gt})\geq 1 - C_He^{-c_Ht}
 \qquad(t\geq i\geq1).
\end{equation}
The same bound holds after averaging over $1\leq i\leq m$ whenever $t\geq m$.
\end{lemma}
\begin{proof}
By \eqref{eq:radial-gaussian} and
\eqref{eq: group Cartan comparison}, the tail outside $G_s$ at time
$N_0i$, for $s\geq C_0'i$, is bounded by
$C\exp(-c(s-C_0i)^2/i)$. The polynomial factors are absorbed by
decreasing $c$ and increasing $C_0'$. Choose
$v_G>\max\{2C_0,C_0'\}$. Then, for $t\geq i$,
$(v_Gt-C_0i)^2/i\geq c't^2/i\geq c't$, proving the assertion.
\end{proof}

\subsubsection{Regularization and density ratios}\label{app:regularization}

To apply the observable lattice criterion to a weakly almost-invariant
measure, we need a positive density with uniformly bounded, Lipschitz
Radon--Nikodym ratios. A fixed smoothing kernel suffices.
Throughout, \(Y:=G/\Gamma\) carries the quotient metric.

\begin{claim}[Regularization by a fixed kernel]\label{cl:regularization}
There is a probability density \(h>0\) on \(G\) and constants
\(0<c_1\leq c_2\), \(L_G,D_G,s_G>0\), depending only on \(G\), with
the following properties. For every probability \(\tau\) on \(Y\), the
measure
\[
 \widetilde\tau:=(h\,dm_G)*\tau
\]
has a positive density \(\widetilde\rho\), and, for \(g\in G_1\),
\[
 \widetilde f_g:=\frac{\widetilde\rho}
                         {\widetilde\rho\circ g^{-1}}
 \quad\text{satisfies}\quad
 c_1\leq\widetilde f_g\leq c_2,\qquad
 \Lip(\widetilde f_g)\leq L_G.
\]
If \(\tau\) is weakly \(\varepsilon\)-invariant, then
\begin{equation}\label{eq:localized-regularization}
 \Wb(\widetilde\tau,\tau)\leq D_G\varepsilon,
 \qquad
 \sup_{g\in G_1}\Wb(\widetilde\tau,g\widetilde\tau)
       \leq D_G\varepsilon.
\end{equation}
Finally, for every \(x\in Y\), \(t>0\) and \(0\leq\delta\leq1\),
\[
 \tau(\overline G_t x)\geq1-\delta
 \quad\Longrightarrow\quad
 \widetilde\tau(\overline G_{t+s_G}x)
       \geq1-\delta-\tfrac18.
\]
\end{claim}

\begin{proof}
Put \(r(g):=d_G(e,g)\). The adjoint norm is bounded by
\(e^{Cr(g)}\) for a group constant \(C\). Factoring \(g\) along a
minimising path into at most \(2r(g)+1\) elements of \(G_1\), and
telescoping the translates of a weakly \(\varepsilon\)-invariant
probability, gives
\begin{equation}\label{eq:global-weak-invariance}
 \Wb(g\tau,\tau)
 \leq C(1+r(g))e^{Cr(g)}\varepsilon.
\end{equation}
Here we use that left translation by \(u\) is
\(\max\{1,\|\Ad(u)\|\}\)-Lipschitz on bounded-Lipschitz distance.
All partial products in the factorisation have distance at most
\(r(g)\) from the identity.

Choose \(a>0\) larger than the sum of the constant in
\eqref{eq:global-weak-invariance} and the exponent in the exponential volume-growth of \(G\). Then $Z:=\int_Ge^{-ar(g)}\,dm_G(g)<\infty$. Set
\[
 h(g):=Z^{-1}e^{-ar(g)}.
\]
Then one sees that \(h>0\) and \(\Lip(\log h)\leq a\). The right translates of
\(h\) satisfy the same inequality. Their periodizations on \(Y\)
can therefore be checked to satisfy
\[
 e^{-a d_Y(x,y)}\rho(y)\leq\rho(x)
                    \leq e^{a d_Y(x,y)}\rho(y).
\]
 In particular,
\(\Lip(\log\widetilde\rho)\leq a\).

For \(g\in G_1\), the orbit-map bound gives
\(d_Y(x,g^{-1}x)\leq1\). Put
\[
 A_{\mathrm{act}}:=\max\{1,\sup_{g\in G_1}\|\Ad(g)\|\}.
\]
We obtain
\[
 e^{-a}\leq\widetilde f_g\leq e^a,\qquad
 \Lip(\log\widetilde f_g)\leq a(1+A_{\mathrm{act}}).
\]
Thus \(c_1=e^{-a}\), \(c_2=e^a\) and
\(L_G=e^a a(1+A_{\mathrm{act}})\) suffice.

The right side of \eqref{eq:global-weak-invariance} is integrable
against \(h\,dm_G\). Integrating gives
\(\Wb(\widetilde\tau,\tau)\leq D_0\varepsilon\).
The triangle inequality then gives, for \(g\in G_1\),
\[
 \Wb(\widetilde\tau,g\widetilde\tau)
 \leq(1+A_{\mathrm{act}})\Wb(\widetilde\tau,\tau)
       +\Wb(\tau,g\tau).
\]
This proves \eqref{eq:localized-regularization} after enlarging \(D_G\).

Finally, choose \(s_G\) so that \(\int_{G_{s_G}}h\,dm_G\geq7/8\).
Right invariance gives
\(G_{s_G}\overline G_t\subseteq\overline G_{t+s_G}\).
The final assertion follows by integrating over
\(G_{s_G}\times\overline G_t x\).
\end{proof}

\section{A Margulis inequality for convolution powers}\label{app: margulis}

\subsection{Standing notation}

Let \(G\) be a connected noncompact simple Lie group with finite centre, and let \(K<G\) be a maximal compact subgroup.  Put
\(X=K\backslash G\) and let \(o=K\in X\).  We write
\[
  \Gamma^g = g\Gamma g^{-1}
\]
for conjugation.  We say that a subset of \(G\) is Zariski dense if its image
under the adjoint representation is Zariski dense in the real algebraic
Zariski closure of \(\Ad(G)\).

We use $m_K$, the Cartan projection $\kappa$, and the Cartan norm
fixed in Section~\ref{sec: def}. In particular,
\begin{equation}\label{eq:distance-kappa-comparable}
 d_X(K,Kg)=\|\kappa(g)\|.
\end{equation}
A probability is bi-$K$-invariant when it is invariant under left and
right multiplication by $K$.

Let \(\Sigma^+\) be the corresponding set of positive restricted roots, and let
\(\Delta\subset\Sigma^+\) be the set of simple restricted roots.  Put
\[
  \uphor=\bigoplus_{\alpha\in\Sigma^+}\g_\alpha .
\]
We also fix the \(\Ad(\exp\aCartan)\)-invariant complementary subspace
\[
  \pminus
  =\g_0\oplus\bigoplus_{\alpha\in\Sigma^+}\g_{-\alpha},
  \qquad \g=\uphor\oplus\pminus,
\]
where \(\g_0\) is the zero restricted-root space.  Let
\(\pru:\g\to\uphor\) denote the projection with kernel \(\pminus\).
For \(H\in\overline{\aCartan^+}\), set
\[
  a_H=\exp H,
  \qquad
  r(H)=\min_{\alpha\in\Delta}\alpha(H).
\]
Thus \(r(H)\ge0\), and \(r(H)>0\) precisely when \(H\) lies in the open Weyl
chamber.

Let \(\|\cdot\|\) be the usual Cartan norm on \(\g=\Lie(G)\).

Let \(\mathcal N\) be the set of all nonzero nilpotent Lie subalgebras of \(\g\), viewed as a subset of the finite disjoint union
\[
  \coprod_{d=1}^{\dim\g}\Grass_d(\g).
\]
For each \(d\), the conditions ``Lie subalgebra'' and ``nilpotent'' are closed
conditions on \(\Grass_d(\g)\); hence \(\mathcal N\) is compact.  It is also
\(\Ad(K)\)-invariant.  For \(\mathfrak n\in\mathcal N\), define
\begin{equation}\label{eq:def-mn}
  m_{\mathfrak n}(g)
  =
  \inf_{0\ne Y\in\mathfrak n}
  \frac{\|\Ad(g)Y\|}{\|Y\|}.
\end{equation}
This is the smallest singular value of \(\Ad(g)\) restricted to
\(\mathfrak n\).

Finally, let \(R>0\) be a Zassenhaus radius for \(G\), chosen small enough that
\(\exp\) is injective on \(B_\g(R)\).  Choose \(0<\rho_*<R\), and set
\[
  V=\exp B_\g(R),
  \qquad
  V_0=\exp B_\g(\rho_*).
\]
For a discrete subgroup \(\Gamma<G\), define
\begin{equation}\label{eq:def-IG}
\Ii_G(\Gamma)=
\begin{cases}
\displaystyle \min_{\gamma\in(\Gamma\cap V_0)\setminus\{e\}}
     \|\log \gamma\|,
  & \Gamma\cap V_0\ne\{e\},\\[1.2ex]
\rho_*,
  & \Gamma\cap V_0=\{e\}.
\end{cases}
\end{equation}
Equivalently,
\[
  \Ii_G(\Gamma)=\sup\{0<r<\rho_*:\Gamma\cap\exp B_\g(r)=\{e\}\}.
\]
We use the Archimedean radius of \cite[\S7]{GLM}. For some $C_{\exp}=C_{\exp}(G)\geq1$, it satisfies
\begin{equation}\label{eq: capped comparison}
 C_{\exp}^{-1}\min\{\rho_*,\Ii(\Gamma)\}\leq\Ii_G(\Gamma)
 \leq C_{\exp}\min\{\rho_*,\Ii(\Gamma)\}.
\end{equation}
This follows from the local comparison of the exponential chart with $d_G$.  

\subsection{The Gelander--Levit--Margulis inputs}

We shall use the following three consequences of Gelander--Levit--Margulis
\cite{GLM}.  They are stated here in the precise form needed below.

\begin{proposition}[GLM inputs]\label{prop:GLM-inputs}
With the notation above, the following hold.
\begin{enumerate}[label=\textup{(GLM\arabic*)}, leftmargin=4.2em]

\item\label{item:GLM-zassenhaus}
\emph{Zassenhaus nilpotent enclosure.}  The radius \(R\) can be chosen so that,
for every discrete subgroup \(\Gamma<G\), the group generated by  $\Gamma \cap V$ is contained in a connected nilpotent subgroup $N(\Gamma).$
When $\Gamma\cap V\ne\{e\}$, choose the enclosure from the
Zassenhaus construction, so that
$\log(\Gamma\cap V)\subseteq\mathfrak n(\Gamma):=\Lie(N(\Gamma))\in\mathcal N$.

\item\label{item:GLM-angular}
\emph{Uniform angular sublevel estimate.}  For the projection
\(\pru:\g\to\uphor\) fixed above, there exist constants
\(C_{\mathrm{ang}}>0\) and \(\theta>0\), depending only on \(G,K,\|\cdot\|\) and
\(\pru\), such that for every \(\mathfrak n\in\mathcal N\) and every
\(0<\varepsilon<1\),
\begin{equation}\label{eq:GLM-angular-sublevel}
  m_K\left\{k\in K:
  \inf_{0\ne Y\in\Ad(k)\mathfrak n}
  \frac{\|\pru Y\|}{\|Y\|}\le\varepsilon\right\}
  \le C_{\mathrm{ang}}\varepsilon^\theta .
\end{equation}

\item\label{item:GLM-linear}
\emph{Linear expansion from angular transversality.}  There is a constant
\(C_{\mathrm{lin}}\ge1\) such that for every
\(H\in\overline{\aCartan^+}\), every \(k\in K\), and every
\(\mathfrak n\in\mathcal N\),
\begin{equation}\label{eq:GLM-linear-expansion}
  m_{\mathfrak n}(a_Hk)
  \ge
  C_{\mathrm{lin}}^{-1} e^{r(H)}
  \inf_{0\ne Y\in\Ad(k)\mathfrak n}
  \frac{\|\pru Y\|}{\|Y\|}.
\end{equation}

\end{enumerate}
\end{proposition}
\begin{proof}
The Zassenhaus enclosure and logarithm inclusion are
\cite[Lemma~7.1 and the proof of Proposition~7.3]{GLM}. For the angular
estimate, apply \cite[Theorem~5.1 and Propositions~6.4--6.6]{GLM}
to the compact family of nilpotent subalgebras of each dimension,
reversing the positive-root convention. Taking the minimum exponent
and maximum constant over the finitely many dimensions gives (GLM2);
enlarging the constant covers all $0<\varepsilon<1$.
Finally, $\pru$ commutes with $\Ad(a_H)$, whose restriction to
$\uphor$ expands by at least $e^{r(H)}$. Since the Cartan norm makes
$\pru$ an orthogonal projection,
\[
 \|\Ad(a_H)Z\|\geq\|\Ad(a_H)\pru Z\|
                    \geq e^{r(H)}\|\pru Z\|.
\]
Apply this to $Z=\Ad(k)Y$ and take the infimum to get (GLM3).
\end{proof}

\subsection{The convolution-power inequality}

\begin{theorem}[Convolution-power GLM inequality]\label{thm:main}
Let \(\lambda\) be a bi-\(K\)-invariant Borel probability measure on \(G\).  Assume
that the closed semigroup generated by \(\supp\lambda\) is Zariski dense in \(G\),
and that \(\lambda\) has exponential moment in the symmetric space, i.e. that for
some \(\alpha>0\),
\begin{equation}\label{eq:exp-moment}
  \int_G e^{\alpha\dX(o,Kg)}\,d\lambda(g)<\infty.
\end{equation}
Then there exist
\[
  N\ge 1, \qquad \delta>0, \qquad 0<c<1, \qquad b<\infty
\]
such that for every discrete subgroup \(\Gamma<G\),
\begin{equation}\label{eq:main-drift}
  \int_G \Ii_G(\Gamma^g)^{-\delta}\,d\lambda^{*N}(g)
  \le c \Ii_G(\Gamma)^{-\delta}+b.
\end{equation}
The constants \(N,\delta,c,b\) may depend on \(\lambda\).  In particular, no claim is
made that \(\delta\) is the explicit group-uniform exponent appearing in
\cite[Theorem 1.5]{GLM}.
\end{theorem}

\begin{remark}[Comparison with GLM]
Gelander--Levit--Margulis prove the corresponding one-step inequality for the
special measure
\[
  \lambda_s=m_K*\delta_s*m_K,
\]
where \(s\in G\) is a sufficiently expanding regular semisimple element; see
\cite[Theorem 1.5]{GLM}.  Their compact support gives a deterministic lower
bound on how much conjugation by an element of \(KsK\) can shrink the
Zassenhaus radius.  In the proof below this deterministic bound is replaced by
an integrable tail estimate, which follows from the exponential-moment
assumption.  Their deterministic expansion of the fixed element \(s\) is
replaced by linear Cartan drift of a high convolution power \(\lambda^{*N}\).
\end{remark}

\subsection{The angular estimate in negative-moment form}

We first turn the angular and linear estimates in Proposition~\ref{prop:GLM-inputs}
into the negative-moment estimate used in the drift argument.

\begin{lemma}[Angular negative moment]\label{lem:angular-negative-moment}
For every \(0<\delta<\theta\) there is a constant
\(C_\delta<\infty\) such that, for every \(\mathfrak n\in\mathcal N\) and every
\(H\in\overline{\aCartan^+}\),
\begin{equation}\label{eq:angular-negative-moment}
  \int_K m_{\mathfrak n}(a_H k)^{-\delta}\,dm_K(k)
  \le C_\delta e^{-\delta r(H)}.
\end{equation}
\end{lemma}

\begin{proof}
By \eqref{eq:GLM-linear-expansion},
\[
  m_{\mathfrak n}(a_Hk)
  \ge C_{\mathrm{lin}}^{-1} e^{r(H)} A_{\mathfrak n}(k),
\]
where
\[
  A_{\mathfrak n}(k)=
  \inf_{0\ne Y\in\Ad(k)\mathfrak n}\frac{\|\pru Y\|}{\|Y\|}.
\]
The angular sublevel estimate \eqref{eq:GLM-angular-sublevel} says that
\[
  m_K\{k:A_{\mathfrak n}(k)\le\varepsilon\}
  \le C_{\mathrm{ang}}\varepsilon^\theta
  \qquad(0<\varepsilon<1),
\]
uniformly in \(\mathfrak n\in\mathcal N\).  Therefore, setting
\[
  Z(k)=C_{\mathrm{lin}}e^{-r(H)}m_{\mathfrak n}(a_Hk),
\]
we have \(m_K\{Z\le\varepsilon\}\le C_{\mathrm{ang}}\varepsilon^\theta\) for
\(0<\varepsilon<1\).  If \(0<\delta<\theta\), then the layer-cake formula gives
\[
\begin{aligned}
  \int_K Z(k)^{-\delta}\,dm_K(k)
  &=\int_0^\infty m_K\{Z^{-\delta}>t\}\,dt  \\
  &\le 1+C_{\mathrm{ang}}\int_1^\infty t^{-\theta/\delta}\,dt
  <\infty .
\end{aligned}
\]
This bound is uniform in \(H\) and \(\mathfrak n\).  Multiplying back by
\(C_{\mathrm{lin}}^\delta e^{-\delta r(H)}\) proves
\eqref{eq:angular-negative-moment}.
\end{proof}

\subsection{Angular averaging and Cartan drift}

Let \(\tau\) be any bi-\(K\)-invariant probability measure on \(G\), not
necessarily compactly supported.  Let \(\overline\tau=\kappa_*\tau\) be its
radial law on \(\overline{\aCartan^+}\).  The Cartan decomposition and
bi-\(K\)-invariance imply the disintegration formula
\begin{equation}\label{eq:KAK-disintegration}
 \tau=\int_{\overline{\mathfrak a^+}}
          m_K*\delta_{a_H}*m_K\,d\overline\tau(H).
\end{equation}
This is disintegration over $K\backslash G/K$ and does not require
compact support.

Since the norm is \(\Ad(K)\)-invariant,
\[
  m_{\mathfrak n}(k_1a_Hk_2)=m_{\mathfrak n}(a_Hk_2)
  =m_{\Ad(k_2)\mathfrak n}(a_H).
\]
As \(\mathcal N\) is \(\Ad(K)\)-invariant, Lemma
\ref{lem:angular-negative-moment} and \eqref{eq:KAK-disintegration}
give, for every \(0<\delta<\theta\),
\begin{equation}\label{eq:any-biK-average}
  \sup_{\mathfrak n\in\mathcal N}
  \int_G m_{\mathfrak n}(g)^{-\delta}\,d\tau(g)
  \le C_\delta
      \int_{\overline{\aCartan^+}}e^{-\delta r(H)}\,d\overline\tau(H).
\end{equation}
The angular estimate is uniform in $H$, so it remains valid after
integration against an unbounded radial law.

We now apply this to \(\tau=\lambda^{*n}\).  Since \(\lambda\) is bi-\(K\)-invariant,
every convolution power \(\lambda^{*n}\) is bi-\(K\)-invariant.  Let
\[
  Y_n=g_n\cdots g_1,
\]
where \(g_1,g_2,\ldots\) are independent \(G\)-valued random variables with law
\(\lambda\).  Then the radial law \(\overline{\lambda^{*n}}\) is the law of
\(\kappa(Y_n)\), and \eqref{eq:any-biK-average} yields
\begin{equation}\label{eq:mn-bound-random-walk}
  \sup_{\mathfrak n\in\mathcal N}
  \int_G m_{\mathfrak n}(g)^{-\delta}\,d\lambda^{*n}(g)
  \le C_\delta
       \E\left[e^{-\delta r(\kappa(Y_n))}\right].
\end{equation}

The exponential moment implies a finite first moment. Applying
Benoist--Quint~\cite[Theorem~9.9(a),(f)]{BQ} in the adjoint group gives
\begin{equation}\label{eq:BQ-LLN}
  \frac{1}{n}\kappa(Y_n)\longrightarrow \sigma_\lambda
\end{equation}
almost surely and in \(L^1\), where \(\sigma_\lambda\in\aCartan^+\) lies in the
open Weyl chamber.  Hence
\[
  \frac{1}{n}r(\kappa(Y_n))\longrightarrow r(\sigma_\lambda)>0
\]
almost surely.  Since \(0\le e^{-\delta r(\kappa(Y_n))}\le1\), dominated
convergence gives
\begin{equation}\label{eq:expectation-to-zero}
  \E\left[e^{-\delta r(\kappa(Y_n))}\right]\longrightarrow0.
\end{equation}
Combining this with \eqref{eq:mn-bound-random-walk}, we obtain the following
conclusion: for every fixed \(0<\delta<\theta\),
\begin{equation}\label{eq:negative-moment-limit}
  \sup_{\mathfrak n\in\mathcal N}
  \int_G m_{\mathfrak n}(g)^{-\delta}\,d\lambda^{*n}(g)
  \longrightarrow 0
  \qquad(n\to\infty).
\end{equation}
The final value of \(\delta\) will be chosen in the next subsection so that it
also satisfies the tail-integrability requirement.

For the heat kernel used in this paper, this choice of convolution power
can be made quantitative. The proof of
Proposition~\ref{prop:heat-kernel-cartan-concentration} applies with
\(N_0=1\), independently of the later choice of \(N_0\).
Since \(2\varrho\) lies in the open Weyl chamber, there are positive
group constants \(v,c,C\) such that, for a walk with step law \(p_1\),
\[
 \mathbb P\{r(\kappa(Y_n))<vn\}\leq Ce^{-cn}.
\]
Consequently
\begin{equation}\label{eq: heat negative moment rate}
 \E e^{-\delta r(\kappa(Y_n))}
 \leq e^{-\delta vn}+Ce^{-cn}.
\end{equation}
For a fixed admissible \(\delta\), it is enough to choose \(N\) so that
\(C_\delta(e^{-\delta vN}+Ce^{-cN})\leq1/2\).
The global heat-kernel estimate also bounds the tail moments below, and
the further power in Lemma~\ref{lem: radius transfer} is then chosen
by its displayed inequality. Thus the heat-kernel application does not
require a convergence rate in the general law of large numbers.

\subsection{The tail moment}

Define
\begin{equation}\label{eq:def-L}
  L(g)=\max\{1,\|\Ad(g^{-1})\|\},
\end{equation}
where the operator norm is taken with respect to the fixed Euclidean norm on
\(\g\).  Since \(\Ad\) is a finite-dimensional representation and
\eqref{eq:distance-kappa-comparable} holds, there is a constant \(B>0\) such
that
\begin{equation}\label{eq:L-kappa-bound}
  \log L(g)\le B(1+\dX(o,Kg))
\end{equation}
for all \(g\in G\).  Therefore the exponential-moment assumption
\eqref{eq:exp-moment} implies that
\begin{equation}\label{eq:tail-moment-one-step}
  \int_G L(g)^\delta\,d\lambda(g)<\infty
\end{equation}
for every sufficiently small \(\delta>0\).  Equivalently, this is the standard
comparison between exponential moments of the Cartan projection and exponential
moments of operator norms in finite-dimensional representations; see
\cite[Chapter 9]{BQ}.

We now choose \(\delta>0\) satisfying both conditions
\begin{equation}\label{eq:choose-delta}
  0<\delta<\theta
  \qquad\text{and}\qquad
  \int_G L(g)^\delta\,d\lambda(g)<\infty.
\end{equation}
With this \(\delta\) fixed, choose \(N\) so large that
\begin{equation}\label{eq:c0-less-than-one}
  c_0:=
  \sup_{\mathfrak n\in\mathcal N}
  \int_G m_{\mathfrak n}(g)^{-\delta}\,d\lambda^{*N}(g)<1,
\end{equation}
which is possible by \eqref{eq:negative-moment-limit}.  Since
\[
  L(gh)\le L(g)L(h),
\]
we also have
\begin{equation}\label{eq:tail-moment-N}
  M_N:=\int_G L(g)^\delta\,d\lambda^{*N}(g)<\infty.
\end{equation}
Indeed, \(M_N\le (\int L(g)^\delta\,d\lambda(g))^N\).

Thus exponential moment supplies the integrable bound needed for the
unbounded support of the convolution power.

\subsection{Pointwise estimates for the discreteness radius}

The next two lemmas replace the pointwise compact-support estimates used in the
proof of \cite[Theorem 1.5]{GLM}.  

\begin{lemma}[Thin-part estimate]\label{lem:thin}
Assume \(\Ii_G(\Gamma)<\rho_*\), and let \(\mathfrak n(\Gamma)\in\mathcal N\) be as above.  Then, for every \(g\in G\),
\begin{equation}\label{eq:thin-lower}
  \Ii_G(\Gamma^g)
  \ge
  \min\left\{\rho_*,
       m_{\mathfrak n(\Gamma)}(g)\Ii_G(\Gamma),
       \frac{R}{L(g)}\right\}.
\end{equation}
Consequently,
\begin{equation}\label{eq:thin-upper-negative}
  \Ii_G(\Gamma^g)^{-\delta}
  \le \rho_*^{-\delta}
     +m_{\mathfrak n(\Gamma)}(g)^{-\delta}\Ii_G(\Gamma)^{-\delta}
     +R^{-\delta}L(g)^\delta.
\end{equation}
\end{lemma}

\begin{proof}
Write \(I=\Ii_G(\Gamma)\), \(m=m_{\mathfrak n(\Gamma)}(g)\), and \(L=L(g)\).  Let
\[
  0<r<\min\left\{\rho_*,mI,\frac{R}{L}\right\}.
\]
Suppose, for contradiction, that \(\Gamma^g\cap\exp B_\g(r)\) contains a
nontrivial element.  Then there is \(\gamma\in\Gamma\setminus\{e\}\) such that
\(g\gamma g^{-1}=\exp X\) with \(\|X\|<r\).  Since
\[
  \gamma=\exp(\Ad(g^{-1})X)
\]
and \(\|\Ad(g^{-1})X\|\le L\|X\|<R\), we have \(\gamma\in\Gamma\cap V\).  By
(GLM1), we have \(\log\gamma\in\mathfrak n(\Gamma)\).  Also
\[
  X=\log(g\gamma g^{-1})=\Ad(g)\log\gamma.
\]
Therefore \(\|X\|\ge m\|\log\gamma\|\).  Since \(\|X\|<mI\), it follows that
\(\|\log\gamma\|<I\).  As \(I<\rho_*\), this places \(\gamma\) in
\(\Gamma\cap V_0\), contradicting the definition of \(I=\Ii_G(\Gamma)\).  Hence
\(\Gamma^g\cap\exp B_\g(r)=\{e\}\) for every such \(r\), proving
\eqref{eq:thin-lower}.  The estimate \eqref{eq:thin-upper-negative} follows by
taking negative powers and using
\((\min\{A,B,C\})^{-\delta}=\max\{A^{-\delta},B^{-\delta},C^{-\delta}\}
\le A^{-\delta}+B^{-\delta}+C^{-\delta}\).
\end{proof}

\begin{lemma}[Thick-part estimate]\label{lem:thick}
If \(\Ii_G(\Gamma)\ge \rho_*/2\), then for every \(g\in G\),
\begin{equation}\label{eq:thick-upper-negative}
  \Ii_G(\Gamma^g)^{-\delta}
  \le \rho_*^{-\delta}+(2/\rho_*)^\delta L(g)^\delta.
\end{equation}
\end{lemma}

\begin{proof}
For arbitrary \(\Gamma\) and \(g\), the same logarithm-equivariance argument as
above gives
\begin{equation}\label{eq:crude-lower}
  \Ii_G(\Gamma^g)
  \ge \min\left\{\rho_*,\frac{\Ii_G(\Gamma)}{L(g)}\right\}.
\end{equation}
Indeed, if \(0<r<\min\{\rho_*,\Ii_G(\Gamma)/L(g)\}\) and
\(g\gamma g^{-1}=\exp X\) with \(\|X\|<r\), then
\[
  \|\log\gamma\|=\|\Ad(g^{-1})X\|<\Ii_G(\Gamma),
\]
contradicting the definition of \(\Ii_G(\Gamma)\).  If
\(\Ii_G(\Gamma)\ge\rho_*/2\), then \eqref{eq:crude-lower} implies
\[
  \Ii_G(\Gamma^g)^{-\delta}
  \le \rho_*^{-\delta}+\Ii_G(\Gamma)^{-\delta}L(g)^\delta
  \le \rho_*^{-\delta}+(2/\rho_*)^\delta L(g)^\delta.\qedhere
\]
\end{proof}

\subsection{Completion of the drift estimate}

\begin{proof}[Proof of Theorem~\ref{thm:main}]
Let $\tau:=\lambda^{*N}$, with $\delta,N,c_0,M_N$ as in
\eqref{eq:choose-delta}--\eqref{eq:tail-moment-N}.
If $\Ii_G(\Gamma)<\rho_*$, integration of
\eqref{eq:thin-upper-negative} gives
\begin{equation}\label{eq:thin-integrated}
 \int_G\Ii_G(\Gamma^g)^{-\delta}\,d\tau(g)
 \leq c_0\Ii_G(\Gamma)^{-\delta}+\rho_*^{-\delta}+R^{-\delta}M_N.
\end{equation}
If $\Ii_G(\Gamma)\geq\rho_*/2$, the thick-part estimate gives
\begin{equation}\label{eq:thick-integrated}
 \int_G\Ii_G(\Gamma^g)^{-\delta}\,d\tau(g)
 \leq\rho_*^{-\delta}+(2/\rho_*)^\delta M_N.
\end{equation}
These two cases cover all subgroups. Taking
\[
 c:=c_0<1,\qquad
 b:=\rho_*^{-\delta}+R^{-\delta}M_N+(2/\rho_*)^\delta M_N
\]
proves \eqref{eq:main-drift}.
\end{proof}

\subsection{Transfer to the Riemannian radius}

\begin{lemma}[Transfer to the Riemannian discreteness radius]
\label{lem: radius transfer}
There are an integer $N_0\geq1$, an exponent $\varkappa>0$, $c_0\in(0,1)$
and $b_0>0$, depending only on $G$, such that $\mu:=p_{N_0}$ satisfies
\begin{equation}\label{eq:margulis-drift}
 \int_G\Ii(\Gamma^g)^{-\varkappa}\,d\mu(g)
 \leq c_0\Ii(\Gamma)^{-\varkappa}+b_0
\end{equation}
for every discrete $\Gamma\leq G$.
\end{lemma}
\begin{proof}
The heat kernel $p_1$ is bi-$K$-invariant and has full support.
The global bound \eqref{eq:sharp-heat-kernel} gives its exponential
moment, so Theorem~\ref{thm:main} applies. Write its exponent as
$\varkappa$ and its constants as $N,c,b$. The comparison
\eqref{eq: capped comparison}, with $C_{\exp}\geq1$, implies
\[
 \Ii(\Gamma)^{-\varkappa}\leq C_{\exp}^\varkappa\Ii_G(\Gamma)^{-\varkappa},
 \qquad
 \Ii_G(\Gamma)^{-\varkappa}
 \leq C_{\exp}^\varkappa\bigl(\Ii(\Gamma)^{-\varkappa}+\rho_*^{-\varkappa}\bigr).
\]
Iterating the capped drift inequality $s$ times and using these two
inequalities gives
\begin{align*}
 \int_G\Ii(\Gamma^g)^{-\varkappa}\,dp_{Ns}(g)
 &\leq C_{\exp}^{2\varkappa}c^s\Ii(\Gamma)^{-\varkappa}\\
 &\quad+C_{\exp}^{2\varkappa}c^s\rho_*^{-\varkappa}
       +C_{\exp}^\varkappa b/(1-c).
\end{align*}
Choose $s$ so that $C_{\exp}^{2\varkappa}c^s\leq1/2$.
Then \eqref{eq:margulis-drift} holds with
\[
 N_0:=Ns,\qquad c_0:=\frac12,\qquad
 b_0:=\max\{1,C_{\exp}^{2\varkappa}c^s\rho_*^{-\varkappa}+C_{\exp}^\varkappa b/(1-c)\}.\qedhere
\]
\end{proof}

\phantomsection
\addcontentsline{toc}{section}{References}
\bibliographystyle{plain}
\bibliography{BibErg}

\end{document}